\documentclass{article}[12pt]
\usepackage{amsmath,amsfonts,amsthm}
\usepackage{amssymb}

\usepackage{epsfig}

\usepackage{subfigure}
\usepackage{graphics}
\usepackage{graphicx}
\usepackage{color}
\usepackage{soul}
\usepackage[makeroom]{cancel}

\usepackage{latexsym}
\usepackage{graphics}
\usepackage{amsmath}
\usepackage{amsthm}
\usepackage{xspace}
\usepackage{amssymb}
\usepackage{color}
\usepackage{cite}
\usepackage{latexsym}
\usepackage{graphics}
\usepackage{amsmath}
\usepackage{amsthm}
\usepackage{xspace}
\usepackage{amssymb}
\usepackage{epsfig}
\usepackage{comment}

\DeclareMathOperator*{\argmax}{arg\,max}

\newcommand{\beql}[1]{\begin{equation}\label{#1}}
\newcommand{\eeql}{\end{equation}}
\newcommand{\eqn}[1]{(\ref{#1})}

\newcommand{\pr}{\mathbb{P}}

\newcommand{\Q}{\mathbb{Q}}

\newcommand{\ch}{{\cal H}}

\newcommand{\cl}{{\cal L}}

\newcommand{\Z}{\mathbb{Z}}

\newcommand{\edit}[1]{#1}

\newtheorem{thm}{Theorem}
\newtheorem{lem}[thm]{Lemma}
\newtheorem{prop}[thm]{Proposition}

\newtheorem{conjecture}[thm]{Conjecture}

\theoremstyle{remark}

\begin{document}

\title{Multi-floor generalization of TASEP 
}

\author
{
Yuliy Baryshnikov \\
Math Department and\\
The Grainger College of Engineering\\
ECE Department
and Coordinated Science Lab\\
University of Illinois at Urbana-Champaign\\
Urbana, IL 61801\\
\texttt{ymb@illinois.edu}\\
\and
Alexander L. Stolyar \\
The Grainger College of Engineering\\
ISE Department
and Coordinated Science Lab\\
University of Illinois at Urbana-Champaign\\
Urbana, IL 61801\\
\texttt{stolyar@illinois.edu}
}

\date{\today}

\maketitle

\begin{abstract}

We consider an interacting particle system, which generalizes the classical totally asymmetric simple exclusion process (TASEP), in that each site can contain up to a fixed finite number of particles, and the particle movement is governed by a {\em back-pressure} (BP) algorithm (also often called {\em MaxWeight}). There are $N$ sites (with $N$ finite or infinite), each may contain at most $c$ particles, $1 \le c < \infty$. New particles enter the system at the left-most site $1$ as a Poisson process of rate $\alpha\le 1$, unless site $1$ has $c$ particles. Particles (if any) are removed from the right-most site $N$ as a Poisson process of rate $\beta \le 1$. The left-to-right movement of particles between neighboring sites is governed by the BP rule: one particle moves from site $n$ to $n+1$ at epochs of a rate $1$ Poisson process, as long as the former site has strictly more particles than the latter. When $c=1$, this is the standard TASEP.

Our main results address the asymptotics of the stationary distribution of a finite system, and especially the limit of the flux (current) as $N\to\infty$. In particular, we prove that interesting non-trivial phase transitions take place in a system with $c>1$.
For example, if $c>1$ and $1/2 \le \beta \le 1$, the maximum limiting flux $1/4$ is achieved as long as $\alpha \ge \alpha_c^*$, where $\alpha_c^* < 1/2$
is some non-trivial threshold. (For the standard TASEP the threshold is $1/2$.) We also put forward a general conjecture about the stationary distribution asymptotics under an arbitrary parameter setting.
We illustrate our formal results and the conjecture by simulations, and identify interesting directions for further research.

\end{abstract}



\section{Introduction}
\label{sec-model-general}

\subsection{Model and motivations}
\label{sec-model-motivations}

We consider an interacting particle system, which generalizes the classical totally asymmetric simple exclusion process (TASEP) \cite{Blythe_2007, Liggett-1975, Liggett-book}. In TASEP each site contains either $0$ or $1$ particles. In our system each site can contain up to a fixed finite number of particles, and the particle movement is governed by a {\em back-pressure} (BP) algorithm \cite{TE92} (also often called {\em MaxWeight}). 

Specifically, the model is as follows. 
A system can be either finite, consisting of $N$ sites labeled by $n\in D_N=\{1,2,\ldots,N\}$, or one-sided, 
with the infinite number of sites, labeled by $n \in D_\infty = \{1,2,\ldots\}$. 
The number of particles (``queue length'') $Q_n(t)$ at site $n$ at time $t\ge 0$  is upper bounded (for all $n$ and $t$)  by 
a fixed finite integer $c\ge 1$. The system operates in continuous time. There is a Poisson process of rate $\alpha \in (0,1]$ of ``arrival bell rings;''
if this bell rings at time $t$ one additional particle is added to (arrives at) the left-most site 1, as long as $Q_1(t-)<c$, and is ``blocked'' otherwise.
There is a Poisson process of rate $1$ of ``bell rings'' for each pair of neighboring sites $(n,n+1)$;
if the bell rings at time $t$ one particle moves form site $n$ to site $n+1$, as long as $Q_n(t-) > Q_{n+1}(t-)$.
In a finite system there is also a Poisson process of rate $\beta \in (0,1]$ of ``departure bell rings;''
if this bell rings at time $t$ one particle is removed (departs the system) from the right-most site $N$, as long as $Q_N(t-)>0$.
All these driving Poisson processes are independent. In the special case of $c=1$, the process is the classical, well-studied TASEP. 
(We could allow cases $\alpha>1$ and/or $\beta>1$ as well -- the results will be essentially same as for $\alpha=1$ and/or $\beta=1$, respectively.
We will not treat such cases explicitly to avoid clogging the exposition.)

It is natural, and useful, to visualize each site as having $c$ ``floors.'' If site $n$ has $k \in \{0, \ldots,c\}$ particles, i.e. $Q_n=k$,
then the particles ``occupy'' floors $1,\ldots, k$ (no floors when $k=0$), one per a floor, and floors $k+1, \ldots,c$ are ``vacant.''
The ``single-floor'' process, with $c=1$, is the classical TASEP.

The BP (MaxWeight) algorithm for a queueing network control, originally introduced in \cite{TE92}, is well-known to have a variety of ``nice'' properties, including ensuring network stochastic stability (if such is feasible at all). One of the important considerations is how a network performance scales under BP, when the network becomes large. Paper \cite{St2011-tandem-queues} addresses the queue scaling under BP for the simple network consisting of the large number $N$ of queues in tandem; specifically, the model in \cite{St2011-tandem-queues} is just like in this paper, with $\beta=1$, but with $c=\infty$, i.e. there is no ``blocking'' of arriving particles at site $1$ and no limit on the queue lengths. 
It is shown in \cite{St2011-tandem-queues} that the following phase transition takes place: if $\alpha>1/4$, the steady-state queues grow to infinity as $N\to\infty$; while if $\alpha < 1/4$, the queues stay $O(1)$. The threshold $1/4$, at which the phase transition occurs, is closely related to the property that $1/4$ is the maximum flux (current) of the one-sided single-floor TASEP. In fact, as shown in \cite{St2011-tandem-queues}, when $\alpha<1/4$,
the process ``lives on the first floor at the sites far on the right,'' and therefore the single-floor TASEP maximum flux $1/4$ is sufficient to handle the exogenous arrival rate $\alpha$.

One motivation for the present work is to consider scaling properties of BP in a network with blocking, i.e. with some limit $c<\infty$ on the maximum queue length. In this case, of course, the queue lengths are always bounded by $c$, but, as $N\to\infty$, there are interesting questions of the limiting system flux (current, throughput) and of the limiting spatial pattern of the queues. These questions have well-known answers for the single-floor TASEP 
(summarized in, e.g., \cite[Table 1]{Blythe_2007}), 
but appear very non-trivial for a multi-floor TASEP that we introduce in this paper. 

Another way to look at our model is to view it as a certain kind of a {\em multi-type} (or, {\em multi-species}) TASEP (see, e.g., \cite{ferraristationary2007} and references therein), which is, however, distinct from the models studied in previous work.\footnote{We remark that our multi-floor model has {\em no} relation
to the {\em multi-line representation} of stationary distributions of some multi-type TASEP models, as in \cite{ferraristationary2007}.}
A site with $i$ particles in it can be viewed 
as occupied by a ``particle'' of type $i \in \{0,1,\ldots,c\}$. When the bell rings for the site pair $(n,n+1)$, 
and the ``particle'' on the left is of strictly higher type, the ``particles'' change their types.
For example, in the 2-floor model ($c=2$), for any site pair away from the boundaries, 
the (swapping) transitions $21 \to 12$, $10 \to 01$, as well as transition $20 \to 11$, are allowed.
This has a rather appealing interpretation of our system in terms of blocks of sites (``zones'') with ``collisions'' of ``particles'' of different types at zone boundaries (fronts). Still using $c=2$ as an example, there will be the (``second-floor'') zone on the left with ``particles'' of types 2 and 1 only,
and with the behavior being exactly that of standard single-floor TASEP, with '2's and '1's moving to the right and left, respectively,
and playing role of particles and holes, respectively. There will be the another (``first-floor'') zone on the right with ``particles'' of types 1 and 0 only,
with the behavior being exactly that of standard single-floor TASEP, with '1's (particles) and '0's (holes) moving to the right and left, respectively.
At the boundary between left and right zones, '2's moving right and '0's moving left ``collide'' and both turn into '1's -- this is an interpretation 
of the transition $20 \to 11$. (One may think of '2's and '0' as being, respectively, positively and negatively charged,
and '1's being ``neutral''. Then transition $20 \to 11$ is interpreted as neutralization of particles of opposite charges upon collision. If we further interpret 
'1's as ``holes,'' then a neutralization is an ``annihilation'' of ``particles" of opposite charges upon collision; see Figure~\ref{fig-annih}.)
For a general $c\ge 2$, going from left to right, there will be a zone on each floor $c, \ldots,1$.  (Each of these zones may be collapsed to zero length.) 
Within the floor-$m$ zone (away from its boundaries) only the swapping transitions $m,m-1 \to m-1,m$ occur.
Conversions $m,m-2 \to m-1,m-1$ occur at the boundary between floor-$m$ and floor-$(m-1)$ zones, due to ``collisions'' of '$m$'s moving to the right with 
'$(m-2)$'s moving to the left.

      \begin{figure}
\centering
        \includegraphics[width=5in]{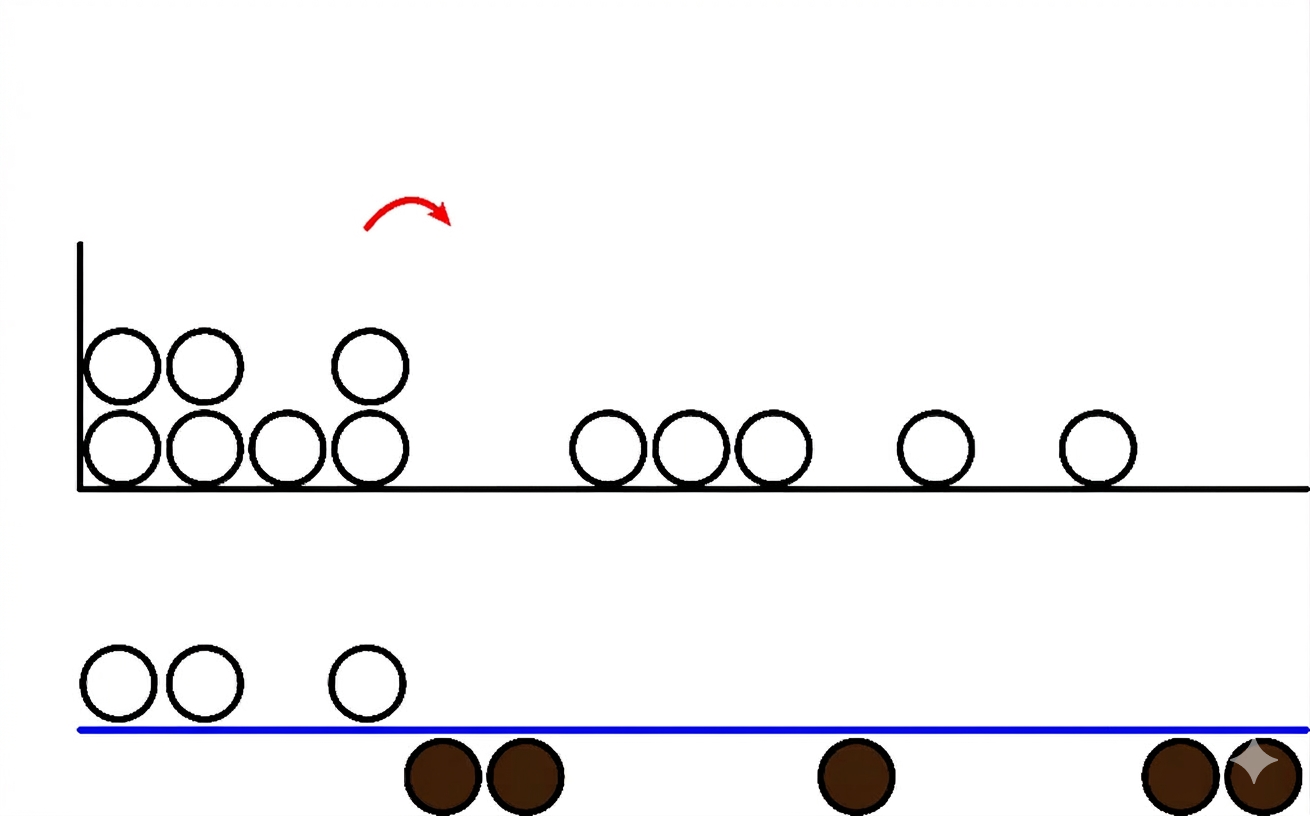}
    \caption{Top: 2-floor TASEP, with all particles moving to the right. Bottom: Interpretation in terms of positive ``particles'' moving to the right and negative "particles" moving to the left. The move shown on the top picture corresponds to the annihilation of the two colliding ``particles'' on the bottom picture.}
    \label{fig-annih}
\end{figure}

\edit{
Yet another related class of models are {\em multi-lane} generalizations of TASEP (cf. \cite{CEKT16} and references therein). In such models, roughly speaking, there are multiple ``lanes,'' spanning a common set of sites. Within each lane particles move as in the
classical single-lane (single-floor) TASEP, but in addition a particle in lane $m$ at the rate $d_{m,\ell}$ hops (migrates)
 to a different, vacant lane $\ell$ within same site. Each lane $m$ has its own attempted arrival and departure rates, $\alpha_m$ and $\beta_m$,
  at the left and right boundaries, respectively. (The model in the present paper is not within this framework.)
 In particular, \cite{CEKT16} studies a mean-field based hydrodynamic model and its stationary solutions, under additional assumptions on the 
 arrival and departure rates.
 Paper \cite{SJHW11} considers a strongly asymmetric two-lane system, in which only the migration 
 from ``higher'' to ``lower'' lane is allowed. 
  The paper uses mean-field approximation to study the system steady-state and
 obtain the approximate phase diagram. 
 For the model in the present paper such mean-field approximation would not lead to any non-trivial results, even approximate. 
 We will elaborate on the comparison of our results to those in \cite{CEKT16,SJHW11} in Section~\ref{sec-discussion}.
 To summarize: the present paper studies a different model from those in multi-lane 
 literature, and it obtains rigorous results for the exact microscopic model; the methods in previous work do not appear to lead to 
 non-trivial conclusions for our model.
 }

\subsection{Setup}

We proceed with necessary definitions and notation.
We can and will adopt a convention that in a finite system $Q_n(t) \equiv 0$ for $n>N$ at all times.
Then, for finite and one-sided systems, $Q(t)=(Q_n(t), ~n\in D_\infty)$, is a Markov process. 
A process state is a sequence 
$Q=(Q_n, ~n\in D_\infty)$, with  $Q_n \in \{0,1,\ldots,c\}$. The order relation $Q \le Q'$ between states means $Q_n \le Q'_n, \forall n$.
The process $Q(t)$ possesses the basic monotonicity property (see \cite[lemma 3]{St2011-tandem-queues}): if we consider two versions of the process, with initial states $Q(0) \le Q'(0)$, they can be coupled so that  $Q(t) \le Q'(t)$ holds at all times $t\ge 0$. 

Obviously, a finite system has finite state space and then the unique stationary distribution. 
The one-sided system may have multiple stationary distributions; however, one of them -- the {\em lower invariant measure} (LIM) -- is unique,
which follows from the above stated basic monotonicity. Specifically, the LIM is the stationary distribution obtained as the $t\to\infty$ limit of distributions of $Q(t)$, given the system starts from ``empty'' state (with all $Q_n=0$), in which case $Q(t)$ is stochastically monotone 
non-decreasing (w.r.t. $\le$ order on the state space). 

It is instructive to introduce non-increasing majorant and minorant, $Q^+_n$ and $Q^-_n$ of a process state $Q$: 
$$ 
Q^-_n=\min_{k\leq n} Q_k;  ~~ Q^+_n=\max_{k\geq n} Q_k.
$$ 
It is easy to observe (see \cite[eq. (2) and lemma 3]{St2011-tandem-queues}) that, under the BP rule, if condition
\beql{eq-queue-cond}
Q^-_n(t) \ge Q^+_{n+1}(t)-1, ~\forall n,
\eeql
holds for some time $t$, then it holds for all times thereafter. Since we will be primarily interested in the LIM of the process, WLOG we can restrict the state space to the states reachable from the empty state (all $Q_n=0$) -- and such states must satisfy 
\eqn{eq-queue-cond}. Therefore, we will consider the process in the state space
$$
\Q^c = \{Q=(Q_n, ~n\in D_\infty)~~|~~Q_n \in \{0,1,\ldots,c\}, ~ Q^-_n \ge Q^+_{n+1}-1, ~\forall n\},
$$
with the topology of coordinate-wise convergence. It will be convenient to consider each $\Q^c$ as a subspace of the common space $\Q = \cup_{c\ge 1} \Q^c$. 

\edit{We will say that a distribution $\ch$ on $\Q$ is {\em dominated} by a distribution $\tilde \ch$, if random elements $Q$ and $\tilde Q$ with these distributions
can be coupled so that $Q \le \tilde Q$. 
Following  \cite{Liggett-1975} we denote by $\nu_\gamma$, $\gamma \in [0,1]$, the i.i.d. Bernoulli (with parameter $\gamma$) distribution on $\Q^1$, or on $\Q$, and say that a distribution on $\Q$ {\em behaves like $\nu_\gamma$ at $\infty$} if the sequence of shifted to the left by $\ell \ge 1$ versions of the distribution weakly converges to $\nu_\gamma$ as $\ell\to\infty$.}

The main question that we address in this paper is the limit of the stationary distribution of a finite system, and especially the limit of the flux (current), as $N\to\infty$. There are different ways to center site locations, as we consider $N\to\infty$ limit. The ``left-side'' centering is when, as usual, we consider $D_N=\{1,\ldots,N\}$ as a subset of 
$D_\infty = \{1,2, \ldots\}$; in other words, the sites retain their left-to-right labeling as $N$ increases. The ``right-side'' centering is when, for each $N$, we relabel sites from right to left, namely site $n$ gets new label $N-n+1$, and consider new labels as a subset of $D_\infty$; under the right-side centering, the holes move in the direction of increasing (new) labels, while particles move in the opposite direction. We will use terms left-side and right-side asymptotics, to refer to the corresponding centering. Of course, both asymptotics yield the same limit of the flux.

\subsection{Overview of results and paper structure}

Our results in Section~\ref{sec-model-one-side} address the LIM of the one-sided system. This LIM is easily seen, 
from basic monotonicity, to be the limit, as $N\to\infty$, of the stationary distribution of the finite system with $N$ sites and $\beta=1$;
\edit{the proof is same as that of \cite[proposition 5(iii)]{St2011-tandem-queues}.}
(In particular, the limit of the steady-state flux of the finite system, as $N\to\infty$, is equal to the flux of the one-sided system under LIM.)
This is a reason for considering the one-sided system, in addition to it being of independent interest.
Main results of Section~\ref{sec-model-one-side} (Theorem~\ref{th-main-res} and Lemma~\ref{lem-phi-props}), in particular, prove that the flux under the LIM in one-sided system, as a function of $\alpha$, has an interesting non-trivial phase transition. For any $c\ge 1$, as $\alpha$ is increasing, the flux increases continuously and attains the maximum possible value $1/4$ at some threshold $\alpha_c^*$. It is well known that for the classical single-floor TASEP $\alpha_1^*=1/2$. We show that for any $c>1$ 
the threshold $\alpha_c^*< 1/2$, the sequence of $\alpha_c^*$ is strictly decreasing, and $\alpha_c^* \downarrow 1/4$ as $c \uparrow \infty$. For a given $c>1$: as $\alpha$ increases from $\alpha_c^*$ 
to $\alpha_{c-1}^*$ the flux stays at $1/4$ and the LIM, while stochastically increasing, remains such that w.p.1 the ``right tail'' of the system state is ``on the first floor'' -- namely $Q_n \le 1$ for all large $n$; when $\alpha \ge \alpha_{c-1}^*$, the LIM is the same as for $c-1$, but it is ``shifted up by 1.''

In Subsection~\ref{sec-coupling-contr} we introduce several coupling constructions for our multi-floor model. 
Although they are very natural, some of them (and especially their use in the proof of Theorem~\ref{th-main-res}) 
involve useful conventions and subtleties, which appear to be new and may be of independent interest.

In Section~\ref{sec-left-right} we prove (Theorem~\ref{th-special-max-flux}) 
both the left- and right-side stationary distribution limits 
in the case when $1/2 \le \beta \le 1$ and the maximum possible limiting flux $1/4$ is achieved.
In particular, if $\alpha \ge \alpha^*_c$,
the left-side asymptotic limit for any $\beta\ge 1/2$ is exactly same as for $\beta=1$,
i.e. equal to the LIM of the one-sided system.

Then, in Section~\ref{sec-simulations} we present simulations illustrating our results, and providing intuitions for 
the multi-floor system behavior under general parameter settings. Using the ``language'' developed in our formal results, 
in Section~\ref{sec-obbc} 
we formulate our general Conjecture~\ref{conj-obbc}
and show that all our simulation results conform with this conjecture. 

Finally, in Section~\ref{sec-discussion} \edit{we discuss in more detail the comparison of our results to those in some of the previous work, and}
point out several interesting future research directions.
In particular, there is an interesting question of the behavior of fronts (boundaries) between ``zones,'' where the process ``lives'' on different floors.

\section{One-sided model}
\label{sec-model-one-side}

\subsection{Definitions, basic properties and main result}
\label{sec-main-res}

Consider the one-sided system. 
Recall that the LIM of this system is the $N\to\infty$ limit of stationary distributions of finite systems with $\beta=1$, under left-side asymptotics.

Denote by $\cl(\alpha,c)$
the LIM for parameters $\alpha$ and $c$. (The process may have other stationary distributions as well.)  
We denote by $Q(\infty)$ the random system state in a specified stationary regime (i.e., the distribution of $Q(\infty)$ is a specified stationary distribution). 
When the process is in a stationary regime, the corresponding (steady-state) flux $\phi$ is the average rate at which particles move into site $n$, exogenously if $n=1$ or from $n-1$ if $n\ge 2$; clearly, $\phi = \alpha \pr\{Q_1(\infty)<c\} = \pr\{Q_n(\infty)> Q_{n+1}(\infty)\}$.
Denote by $\phi(\alpha,c)$ the flux under the LIM $\cl(\alpha,c)$. \edit{From classical single-floor TASEP results \cite{Liggett-1975},
we know that: $\cl(\alpha,1) = \nu_{\alpha}$ when $0 \le \alpha \le 1/2$; $\cl(\alpha,1)$ behaves like $\nu_{1/2}$ at infinity when $\alpha > 1/2$;
$\phi(\alpha,1) = \gamma(1-\gamma)$ with $\gamma = \min\{\alpha,1/2\}$.}

For a distribution $\ch$ on $\Q$, we will denote by $\Psi_m \ch$ the corresponding distribution ``shifted up'' by $m \ge 0$, i.e. with the values of $Q_n$ increased by $m$. \edit{Note that $\cl(\alpha,c)$ is dominated by $\Psi_m \cl(\alpha,c-m)$ for any $m=0,1,\ldots, c-1$. 
(Because these measures are the limits of the distribution of $Q(t)$, starting in two initial states, one the empty state, and another the state with all sites having exactly $m$ particles.)
}

For a given state $Q(t)$ we define its projections $B_m(t)$ and $A_m(t)$, $m=1,\ldots,c$, as follows:
$$
B_m(t) = \sup\{n\ge 1 ~|~ Q_n(t) \ge m\} \le +\infty, ~~~ A_{m}(t) = \sup\{n\ge 1 ~|~ Q_k(t) \ge m, ~\forall k \le n\} \le +\infty,
$$
where, by convention, $B_m(t)=0$ or $A_m(t)=0$ if the set in the corresponding definition is empty. (See Figure~\ref{fig-anb}.)
Note that if, for a fixed $m \ge 2$, $B_m(t)=\infty$, then
by \eqn{eq-queue-cond} we have $Q_n(t) \ge m-1$ for all $n$ -- in other words the state is such that the floors $1, \ldots, m-1$ at all sites are occupied -- which means $A_{m-1}(t)=\infty$. Obviously, for any $\alpha>0$ and $c$, under LIM $\cl(\alpha,c)$, $B_1(\infty)=\infty$ a.s.

      \begin{figure}
\centering
        \includegraphics[width=5in]{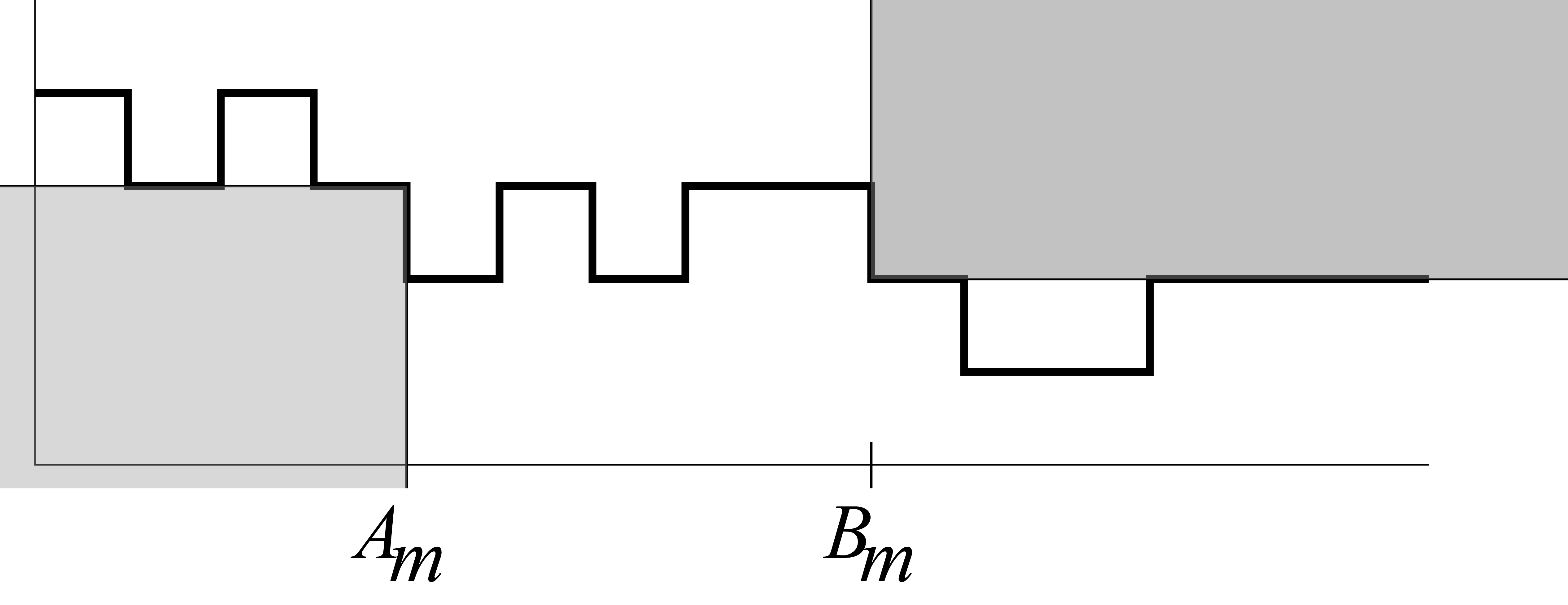}
    \caption{The projections $A_m, B_m$ are shown for $m=3$.}
    \label{fig-anb}
\end{figure}

Note that for any $m\ge 2$, $\pr\{B_{m}(\infty)=\infty\} =\pr\{A_{m-1}(\infty)=\infty\}$.
It is not hard to see that, under LIM $\cl(\alpha,c)$, for each fixed $m \ge 2$, $\pr\{B_{m}(\infty)=\infty\} =\pr\{A_{m-1}(\infty)=\infty\}$ is either $0$ or $1$.
Indeed, suppose  $\pr\{A_{m-1}(\infty)=\infty\} = \delta >0$. The process is such that the state sub-space $\{A_{m-1}(\infty)=\infty\}$ is invariant --
if the process starts in it, it can never leave it. The LIM $\cl(\alpha,c)$ dominates the distribution of the initial state which
is empty with probability $1-\delta$ and has exactly $m-1$ particles in each site with probability $\delta$. 
Therefore, LIM $\cl(\alpha,c)$ is the distributional time limit of the process starting from such initial state.
But then $\delta=\pr\{A_{m-1}(\infty)=\infty\} \ge \delta + (1-\delta)\delta$, which implies $\delta=1$.

We will say that a system in a stationary regime {\em has effective floor $1\le m \le c$} if $\pr\{B_{m}(\infty)=\infty\} =\pr\{A_{m}(\infty)<\infty\}=1$. This means that, a.s., all floors below $m$ are completely occupied, $Q_n$ is either $m$ or $m-1$ for all sufficiently large $n$, and $Q_n$ has each of these values ($m$ and $m-1$) at infinitely many sites. The properties described just above show that, for any $\alpha$ and $c$, LIM $\cl(\alpha,c)$ has a well-defined effective floor $m=m(\alpha,c)$.

Now, it is not difficult to see that $\phi(\alpha,c) \le 1/4$. Indeed, let $m$ be the effective floor of LIM $\cl(\alpha,c)$. Then the floors below $m$ are full and can be ignored as far as flux is concerned, in the sense that it is same as under LIM $\cl(\alpha,c-m+1)$ which has effective floor $1$. For a system with effective floor $1$, the argument, for example, in \cite{St2011-tandem-queues} shows that the flux cannot be greater than $1/4$; so,
$\phi(\alpha,c) \le 1/4$.

Next, it is easy to see that $\phi(\alpha,c)$ 
is non-decreasing Lipschitz continuous in $\alpha$ and  is non-decreasing 
in $c$. (This follows from coupling constructions in Sections~\ref{sec-alpha-coupling} and \ref{sec-c-coupling}.)
Obviously, $\phi(0,c)=0.$ For each $c \ge 1$, define
$$
\alpha_c^* = \min \{\alpha >0 ~|~ \phi(\alpha,c)=1/4\}.
$$
Because when $c=1$ the process is the classical (single-floor) TASEP, we know \cite[theorems 1.8 and 1.7]{Liggett-1975}
that $\alpha_1^*=1/2$, $\phi(\alpha,1) = \alpha(1-\alpha)$ for 
$\alpha \le \alpha_1^*=1/2$, and $\phi(\alpha,1) = 1/4$ for 
$\alpha \ge \alpha_1^*=1/2$. 
By $\phi(\alpha,c)$ monotonicity in $c$, the sequence $\alpha_c^*$ is non-increasing. So, for any $c\ge 1$, we have:
$\phi(\alpha,c)$ is Lipschitz continuous non-decreasing, with $\phi(\alpha,c)=1/4$ for $\alpha \ge \alpha_c^*$.

Our main result for the one-sided system is the following 

\begin{thm}
\label{th-main-res}

(i) The sequence $\alpha_1^*=1/2, \alpha_2^*, \ldots $ is strictly decreasing and 
$\alpha_c^* \downarrow 1/4$.

(ii) For any $c\ge 2$, under LIM $\cl(\alpha,c)$, we have the following. For any $\alpha < \alpha_{c-1}^*$ the effective floor is $1$;
and for any $\alpha \ge \alpha_{c-1}^*$, $\cl(\alpha,c)=\Psi_1 \cl(\alpha,c-1)$. 
Consequently: if $\alpha < \alpha_c^*$, then the effective floor is $1$ and $\phi(\alpha,c)< 1/4$; 
if $\alpha \in [\alpha_m^*,\alpha_{m-1}^*)$ for a fixed $1 \le m \le c$, then the effective floor is $c-m+1$, $\phi(\alpha,c)=1/4$, and $\cl(\alpha,c)=\Psi_{c-m} \cl(\alpha,m)$.
(Here we use convention that $\alpha_0^*=\infty$.)
\end{thm}

The proof of Theorem~\ref{th-main-res} is in Sections~\ref{sec-main-proof-i}-\ref{sec-main-proof-ii}.

Essentially as a corollary of the proof of Theorem~\ref{th-main-res}, 
and of the results in \cite{St2011-tandem-queues}, 
we obtain Lemma~\ref{lem-phi-props} below, which summarizes properties of function $\phi(\alpha,c)$.

 For an $\alpha>1/4$, denote by $c_\alpha$ the minimum number $c$ of floors, for which $\phi(\alpha,c)=1/4$:
$c_\alpha \doteq \min\{c \ge 1 ~|~ \phi(\alpha,c)=1/4\} \equiv \min\{c \ge 1 ~|~ \alpha \ge \alpha^*_{c}\}$.
We will use convention $c_\alpha=\infty$ for $0<\alpha \le 1/4$, so that $c_\alpha$ is defined for all $\alpha>0$.

\begin{lem}
\label{lem-phi-props}
Function $\phi(\alpha,c)$ is strictly increasing, concave, Lipschitz continuous in 
  $\alpha$ in $[0,\alpha^*_c]$. For a fixed $\alpha >1/4$, $\phi(\alpha,c)$ is strictly increasing in $1 \le c \le c_\alpha$, and $\phi(\alpha,c)=1/4$ for $c \ge c_\alpha$.
   For a fixed $\alpha \le 1/4$, $\phi(\alpha,c)$ is strictly increasing in $c \ge 1$,  and $\phi(\alpha,c) \uparrow \alpha$ as $c\uparrow \infty$.
\end{lem}

Proof of Lemma~\ref{lem-phi-props} is in Section~\ref{sec-phi-props-proof}.

\subsection{Coupling constructions}
\label{sec-coupling-contr}

\subsubsection{Coupling of processes under back-pressure algorithm with particles of different priority classes}
\label{sec-priority-coupling}

Suppose we have our system, but with particles belonging to different priority classes, say classes $1, \ldots, \ell$, listed in the order of decreasing priority.
We specify the coupling construction of such a system, which has the following property. If we only ``view'' the evolution of class-1 (highest priority) particles,
this evolution complies with the back-pressure rules. If we only ``view'' the evolution of the set of particles combining classes 1 and 2, this evolution 
also complies with the back-pressure rules. Similarly for combined classes 1, 2, and 3; and so on. The construction will be concerned only with the movement of particles between sites, not exogenous arrivals of particles to site 1 -- an arrival mechanism will be specified later for each concrete system version that we will consider. 

We will use notation $Q_n^{(k)} \ge 0$ for the number of particles of class $k$ at site $n$, and $Q_n^{(\le k)} = \sum_{i \le k} Q_n^{(i)}$ for the total
number of particles of classes $1, \ldots, k$ at site $n$. (It is possible to have more than one particle of a given class at a site, i.e. 
$Q_n^{(k)} \in \{0,1,2, \ldots\}$.)

Consider first a system with two classes, 1 and 2, to simplify exposition. 
When the bell rings for the site pair $(n,n+1)$, the following sequence of events occurs: 1) If $Q_n^{(1)} > Q_{n+1}^{(1)}$,
one class-1 particle moves from $n$ to $n+1$; given that, if site $n+1$ had class-2 particles, then one of them moves back to site $n$. 2) All variables updated. 3) If new value of $Q_n^{(\le 2)}$ is same as the old one \edit{(which occurs when in step 1 either particles did not move at all or one class-1 and one class-2 particles exchanged places)}
and $Q_n^{(\le 2)} > Q_{n+1}^{(\le 2)}$, then one class-2 particle moves from $n$ to $n+1$. It is easy to see that both systems consisting 
of class-1 particles only and consisting of class 1 and 2 combined, comply with the back-pressure rule.

Extension to arbitrary number of classes is straightforward: when the bell rings for the site pair $(n,n+1)$,
sequentially for $k=1,\ldots,\ell$, if the current value of $Q_n^{(\le k)}$ is the same as before the bell ring
and $Q_n^{(\le k)} > Q_{n+1}^{(\le k)}$ we move one class-$k$ particle from $n$ to $n+1$;
if we do move a class-$k$ particle to $n+1$, then one particle of the smallest class $i > k$, present at $n+1$ (if any!), moves back to site $n$.

It is sometimes useful to visualize particles of a higher priority (smaller index) class as occupying lower floors than particles of lower priority (larger index) class. Namely, at site $n$, class-1 particles occupy floors $1,\ldots, Q_n^{(\le 1)}$, class-2 particles occupy floors $Q_n^{(\le 1)}+1, \ldots, Q_n^{(\le 2)}$,
and so on.

\subsubsection{Coupling of processes with different rates $\alpha$}
\label{sec-alpha-coupling}

We now describe coupling of $(\alpha,c)$- and $(\hat \alpha,c)$-systems, $\alpha < \hat \alpha$,
complying with back-pressure rules. The coupled process will have particles of higher and lower priority, of class 1 and 2, respectively
(in the sense of Section~\ref{sec-priority-coupling}). Higher priority, class 1 particles model those in the $(\alpha,c)$-system. They try to enter the system exogenously, at site $1$ at rate $\alpha$; as usual, they do enter
when $Q_1^{(1)} < c$, and are blocked and leave otherwise; if a class $1$ particle does enter site $1$ while $Q_1^{(\le 2)} = c$,
one class-2 particle is removed (``forced back'') from site 1 and leaves the system. There is also the additional independent  
Poisson process, with rate 
$(\hat\alpha-\alpha)$, of exogenous arrival attempts (arrival bell rings) of lower priority, class-2 particles; they actually do enter site $1$ if $Q_1^{(\le 2)} < c$, and are blocked and leave otherwise. The movement of particles between the sites is as specified in Section~\ref{sec-priority-coupling}.
(Note that, under described coupling, it is possible to have more than one particle of a given class at a site.)

Clearly, if we only ``view'' class-1 particles, they describe the evolution of $(\alpha,c)$-system. The evolution of all particles, class 1 and 2 combined,
is that of $(\hat \alpha,c)$-system. If we start this coupled process from the empty state, its distribution converges to the LIM, whose projections on the class-1 only and combined class 1 and 2 processes, are $\cl(\alpha,c)$ and $\cl(\hat\alpha,c)$, respectively. 
When we refer to steady-state in this section, we refer to this LIM. 

Note the following. The steady-state rate at which class-1 particles actually enter site 1 is exactly $\phi(\alpha,c)$, because 
class 1 particles entering the system never leave it. Class-2 particles, on the other hand, can move not only forward but also backward, and can leave the system (when they are forced out from site 1). If $r_{2,in} \le \hat\alpha-\alpha$ is the steady-state rate at which class-2 particles exogenously enter site 1, and 
$r_{2,out} \le r_{2,in}$ is the steady-state rate at which they leave the system, then $r_{2,in} - r_{2,out}$ is the contribution of class-2 particles
into the flux of combined -- classes 1 and 2 -- system. In other words, $\phi(\hat\alpha,c)=\phi(\alpha,c) + r_{2,in} - r_{2,out}$. 
(The fact that $r_{2,out} \le r_{2,in}$ is immediate from the coupling construction, 
because the steady-state refers to the LIM reached in the limit, starting from the empty state.)
In particular, $\phi(\hat\alpha,c)-\phi(\alpha,c) \le r_{2,in} \le \hat\alpha-\alpha$, and so $\phi(\alpha,c)$ is non-decreasing Lipschitz in $\alpha$.

\subsubsection{Coupling of processes with different number $c$ of floors}
\label{sec-c-coupling}

In this section we describe coupling of $(\alpha,c-1)$- and $(\alpha,c)$-systems, complying with back-pressure rules. 
The coupled process will have particles of higher and lower priority, of class 1 and 2, respectively
(in the sense of Section~\ref{sec-priority-coupling}). Class-1 only particles will describe the evolution of $(\alpha,c-1)$-system,
while the combined class-1 and -2 particles will describe $(\alpha,c)$-system. 
Particles try to enter site 1 exogenously as a Poisson process (of arrival bell rings) of rate $\alpha$. 
When a particle tries to enter site 1 (i.e., when the arrival bell rings), then: if $Q_1^{(1)} < c-1$, it will enter site 1 as a class-1 particle,
and if this occurs when $Q_1^{(\le 2)} = c$, one class-2 particle is removed from site 1;
if $Q_1^{(1)} = c-1$ and $Q_1^{(2)} = 0$, then it will enter site 1 as a class-2 particle; 
otherwise, the particle attempting to enter is blocked. The movement of particles between the sites is as specified in Section~\ref{sec-priority-coupling}.
Clearly, if we only ``view'' class-1 particles, they describe the evolution of $(\alpha,c-1)$-system. The evolution of all particles, class 1 and 2 combined,
is that of $(\alpha,c)$-system. (Note that, under described coupling, it is possible to have more than one particle of a given class at a site.)

It will be convenient to adopt the following convention. Suppose there is an exogenous arrival attempt (arrival bell ring) at site 1 when $Q_1^{(1)} < c-1$
and $Q_1^{(\le 2)} = c-1$, in which case $Q_1^{(1)}$ increases by 1 and $Q_1^{(\le 2)}$ becomes $c$. In such a case, although
$Q_1^{(2)}$ does not change, we will assume that from site 1 one class-2 particle is removed (from floors $1,\ldots,c-1$)
and another class-2 particle is immediately added (on floor $c$). With this convention:
a new class-2 particle is added to site 1 when and only when an arrival bell rings and $Q_1^{(\le 2)} = c-1$,
i.e. when in the combined (class 1 and 2) system a new particle is added on floor $c$;
a class-2 particle is removed from site 1 
when and only when an arrival bell rings and $Q_1^{(1)} < c-1$ and $Q_1^{(\le 2)} \ge c-1$.

If we only ``view'' class-1 particles, they describe the evolution of $(\alpha,c-1)$-system. The evolution of all particles, class 1 and 2 combined,
is that of $(\alpha,c)$-system. If we start this coupled process from the empty state, its distribution converges to the LIM, whose projections on the class-1 only and combined class 1 and 2 processes, are $\cl(\alpha,c-1)$ and $\cl(\alpha,c)$, respectively. 
In this section the steady-state refers to this LIM. 

Clearly, $\phi(\alpha,c)=\phi(\alpha,c-1) + r_{2,in} - r_{2,out}$, 
where $r_{2,in}$ is the steady-state rate
at which class-2 particles are added at site 1, and $r_{2,out}$
 is the steady-state rate at which class-2 particles are removed from at site 1. 
 We have $r_{2,out} \le r_{2,in}$ 
 (which is immediate from the coupling construction, 
because the steady-state refers to the LIM reached in the limit, starting from the empty state.)
In particular, $\phi(\alpha,c) \ge \phi(\alpha,c-1)$.
Note that $r_{2,in} = \alpha \pr\{Q_1^{(\le 2)}(\infty) = c-1\}$  and $r_{2,out}= \alpha \pr\{Q_1^{(1)}(\infty) < c-1, Q_1^{(\le 2)}(\infty) \ge c-1\}$.

\subsubsection{Coupling of processes with one being ``impeded''  with respect to another}
\label{sec-ahead}

We will need one additional monotonicity property, which is a version of \cite[proposition 7]{St2011-tandem-queues}. It basically says that if we have two versions of the process, such that the movement of particles in the first one is ``impeded'' w.r.t. the movement of particles in the second, then particles in the first one will always stay ``behind." For completeness, we state the property that we need as Lemma~\ref{lem-ahead} below, and give a proof.

For a process state $Q$, denote by $\|Q\| = \sum_n Q_n$ the total number of particles. So, if $\|Q\| < \infty$, then only sites from a finite subset are occupied.
Additional notation: the partial order relation
$Q \preceq Q'$ means that both $\|Q\| < \infty$ and $\|Q'\|<\infty$, and 
$$
\sum_{k\ge n} Q_k \le \sum_{k\ge n} Q'_k, ~~\forall n\ge 1.
$$

\begin{lem}
\label{lem-ahead}
Consider a fixed realization of process $Q'(\cdot)$, with $\|Q'(0)\| < \infty$, 
complying with the BP rule, under fixed realizations of the bell ring processes for pairs $(n,n+1)$ for all $n \ge 1$, and a fixed realization of the exogenous arrival process at site $1$, which counts particles {\em actually entering} site $1$. 
Assume that no more than one epoch of the above driving realizations can occur at a time, and the realization $Q'(\cdot)$ is well-defined in that 
only a finite number of particle arrivals/movements occur in any finite time interval.
Consider another fixed process realization $Q(\cdot)$, also complying with BP and under the same realizations
of bell rings for site pairs $(n,n+1)$, but such that $Q(0) \preceq Q'(0)$
and with the exogenous arrivals realization being ``delayed'' w.r.t. to $Q'(\cdot)$, namely with the total number of arrivals by any time $t$ being not greater.
Then $Q(t) \preceq Q'(t), \forall t$.
\end{lem}

{\em Proof} is by induction in time. Suppose $Q(t) \preceq Q'(t)$ holds for all times $t$ before time $\tau$ of either exogenous arrival or 
$(n,n+1)$-bell ring. So, we have $Q(\tau-) \preceq Q'(\tau-)$. Let us show that $Q(\tau) \preceq Q'(\tau)$. (We use convention that the realizations are right-continuous.) This is obviously true if $\tau$ is a time of an exogenous arrival. So, suppose at $\tau$ the $(n,n+1)$-bell rings. Suppose, $Q(\tau) \preceq Q'(\tau)$ does {\em not} hold. This can only be the case if a particle moves from $n$ to $n+1$ in $Q(\cdot)$, but not in $Q'(\cdot)$, and $\sum_{\ell \ge n+1} Q_{\ell}(\tau-) = \sum_{\ell \ge n+1} Q'_{\ell}(\tau-)$. But the latter equality, along with $Q(\tau-) \preceq Q'(\tau-)$, implies that $Q_{n}(\tau-) \le Q'_{n}(\tau-)$ and
$Q_{n+1}(\tau-) \ge Q'_{n+1}(\tau-)$, which in turn means that, according to BP rule, a particle has to move from $n$ to $n+1$ in $Q'(\cdot)$ as well.
The contradiction completes the proof.
$\Box$

Note that Lemma~\ref{lem-ahead} is a deterministic property -- it holds for process realizations. This makes it applicable
to situations where driving processes (driving exogenous arrivals and/or particle movement between sites) are {\em not} necessarily independent. 
We also remark that Lemma~\ref{lem-ahead} holds even more generally -- realization $Q(\cdot)$ may be additionally ``impeded'' by disabling (ignoring) some of the $(n,n+1)$-bell rings for it, but not for $Q'(\cdot)$ (see \cite[proposition 7]{St2011-tandem-queues}); same proof applies.

\subsection{Auxiliary lemma}
\label{sec-strict-dom}

Throughout the paper the following fact is used several times.
\edit{
\begin{lem}
\label{lem-strict-dom}
Consider one-sided system with $c\ge 1$ and $\alpha \in (0,1]$. 
Suppose the process $Q(\cdot)$ has two stationary distributions (invariant measures), $\tilde \ch$ and $\ch$, such that $\tilde \ch$ dominates
$\ch$.\\
(i) If $\tilde \ch$  is not equal to $\ch$,  
then \\ $\pr\{Q_1(t) = c\}$ under $\tilde \ch$ is strictly larger than under $\ch$ (and then $\tilde \ch$ has strictly smaller flux).\\
(ii) If $\tilde \ch$ and $\ch$ have equal flux, then $\tilde \ch = \ch$.
\end{lem}
}

\edit{
  {\em Proof of Lemma~\ref{lem-strict-dom}.} Statement (ii) is a corollary of (i), so only (i) needs a proof.
Consider two stationary versions of the system process, $\tilde Q(\cdot)$ and $Q(\cdot)$,
corresponding to the (stationary) distributions $\tilde \ch$ and $\ch$, respectively. These two processes can be coupled so that 
$\tilde Q(t) \ge Q(t)$ at all times, and there exists $\ell\ge 1$ such that $\pr \{\tilde Q_\ell(t) > Q_\ell(t)\} > 0$.
Consider the smallest such $\ell$. 
}

\edit{
If $\ell=1$, then $\pr \{\tilde Q_1(0) > Q_1(0)\} > 0$, and considering the coupled processes evolution over (any) fixed time interval $[0,t]$, with $t>0$, we see that $\pr \{\tilde Q_1(t) =c  > Q_1(t)\} > 0\}$. (It suffices to consider the event when in $[0,t]$
there are exactly $c-\tilde Q_1(0)$ exogenous arrival bell rings at site $1$ and no bell rings for site pair $(1,2)$.) This implies lemma conclusion.
}

\edit{
 If $\ell \ge 2$, then 
$\pr \{\tilde Q_n(0) = Q_n(0),~ n < \ell; ~ \tilde Q_\ell(0) > Q_\ell(0)\} > 0$.
Then, fixing any $t>0$, we see that $\pr \{\tilde Q_n(t) = Q_n(t)=c,~ n < \ell; ~ \tilde Q_\ell(t) = c > Q_\ell(t)\} > 0$.
(Here we easily construct a positive probability event which ensures that the specified condition at time $t$ will hold.)
Then, at any time $t_2 > t$ we will have 
$\pr \{\tilde Q_{\ell-1}(t_2) = c > Q_{\ell-1}(t_2)\} > 0$.
(It suffices to consider the event when in $[t,t_2]$ there is exactly one bell ring for site pair $(\ell-1,\ell)$.)
But this leads to the contradiction with the minimality of $\ell$ and, therefore, the case $\ell \ge 2$ is impossible.
  $\Box$
  }

\subsection{Proof of Theorem~\ref{th-main-res}(i)}
\label{sec-main-proof-i}

\begin{lem}
\label{lem-blocking-prob}
For any $c\ge 1$ and $\alpha$, $\phi(\alpha,c) < \alpha$. Therefore, under LIM $\cl(\alpha,c)$, the probability that a new particle attempting to enter site 1 is ``blocked'' is $(\alpha-\phi(\alpha,c))/\alpha$.
\end{lem}

{\em Proof} follows from the fact that, obviously, $\pr\{Q_1(\infty)=c\}>0$. $\Box$

We see that $\phi(1/4,c) < 1/4$, and therefore $\alpha_c^* > 1/4$ for all $c\ge 1$.

\begin{lem}
\label{lem-rate-to-floor-c}
For a given $c \ge 2$ and $\alpha>0$, consider the coupled processes for $(\alpha,c-1)$- and $(\alpha,c)$-systems, as specified in 
Section~\ref{sec-c-coupling}. 

(i) For any $c\ge 2$, under LIM $\cl(\alpha,c)$, the rate $r_{2,in}(\alpha,c)$, as a function of $0 \le \alpha \le 1/2$, is lower bounded by 
a continuous function $s(\alpha,c)$, which is positive for $0<\alpha \le 1/2$.

(ii) 
For a fixed $1/4 < \alpha \le 1/2$, under LIM $\cl(\alpha,c)$, $\liminf_{c\to\infty} r_{2,in}(\alpha,c) >0$.
\end{lem}

{\em Proof.} (i) The steady-state probability that in the $(\alpha,c)$-system floor $c$ of site 1 is vacant
is at least $1-\alpha$: $\pr\{Q_1^{(\le 2)}(\infty) \le c-1\} \ge 1-\alpha$.
\edit{($\cl(\alpha,c)$ is dominated by $\Psi_{c-1} \cl(\alpha,1)$, and $\cl(\alpha,1)=\nu_\alpha$ for $\alpha \le 1/2$.)}
Consider the stationary version of the process
(with stationary distribution being the LIM). Fix any $\tau>0$. Let $\gamma(\alpha,c)$ be the minimum, over $i=0,1,\ldots,c$, of the probabilities of the following events: in the interval $[0,\tau]$ there are exactly $i$ attempted exogenous arrivals to site $1$, and no attempts to move a particle from 
site 1 to site 2. Clearly, $\gamma(\alpha,c)$, as a function of $\alpha$, is continuous and positive for $0<\alpha \le 1/2$.
We obtain that $r_{2,in}(\alpha,c) = \alpha \pr\{Q_1^{(\le 2)}(\tau) = c-1\} \ge \alpha(1-\alpha)\gamma(\alpha,c)$.

(ii) 
Denote by $X_n(\infty)=c-Q_n^{(\le 2)}(\infty)$ the vacancy of site $n$ in $(\alpha,c)$-system. 
Consider site 1 vacancy $X_1(\infty)$. 
Uniformly in all $c \ge 1$, we have: $\pr\{X_1(\infty) \ge 1\} \ge 1-\alpha$;
$\pr\{X_1(\infty) \le 1\} \ge (\alpha-\phi(\alpha,c-1))/\alpha \ge (\alpha-1/4)/\alpha >0$. We need to show that
\beql{eq-vacancy-0}
\liminf_{c\to\infty} \pr\{X_1(\infty) = 1\}>0.
\eeql

\edit{We will prove \eqn{eq-vacancy-0} by contradiction.
Suppose $\liminf_{c\to\infty} \pr\{X_1(\infty) = 1\}=0$. 
Then we can find a subsequence of $c$ along which $X(\infty) \Rightarrow \tilde X$, where $\tilde X = (\tilde X_n, ~n\ge 1)$ is a random element 
taking values in $\bar \Z_+^{\infty}$, $\bar \Z_+ = \{0,1,2,\ldots\} \cup \{+\infty\}$ is one-point compactification of the set of non-negative integers
and $\Rightarrow$ means convergence in distribution, and moreover, for some fixed
$0<\delta<1$, $\pr\{\tilde X_1 =0\} = \delta$ and $\pr\{\tilde X_1 =1\} = 0$. Then, we necessarily have $\pr\{\tilde X_1 =k\} = 0$ for any $k\ge 1$;
because if not, $\lim_{c\to\infty} \pr\{X_1(\infty) = k\}>0$, which easily implies that $\lim_{c\to\infty} \pr\{X_1(\infty) = 1\}>0$. 
Therefore, 
\beql{eq-vacancy-1}
\pr\{\tilde X_1 =0\} = \delta, ~~ \pr\{\tilde X_1 =\infty\} = 1-\delta.
\eeql
Property \eqn{eq-vacancy-1} means that as $c\to\infty$ the distribution of vacancy $X_1(\infty)$ converges to the one taking value $0$ with probability $\delta$ and $+\infty$ with 
probability $1-\delta$.
But then, necessarily, the stronger property holds:
\beql{eq-vacancy-2}
\pr\{\tilde X_n =0, ~n \le \ell\} = \delta, ~~ \pr\{\tilde X_n =\infty, ~n \le \ell\} = 1-\delta, ~~~\forall \ell\ge 1.
\eeql
Property \eqn{eq-vacancy-2} is proved by induction. Suppose it holds for some $\ell$. Conditional on the event $\tilde X_n =\infty, ~n \le \ell$,  
event $\tilde X_{\ell+1} =\infty$ holds w.p.1 (from \eqn{eq-queue-cond}). Conditional on the event $\tilde X_n =0, ~n \le \ell$,  
event $\tilde X_{\ell+1} =0$ must hold w.p.1, because otherwise, recalling that distribution of $\tilde X$ is a limit of stationary distributions of $X(\cdot)$,
$\pr\{\tilde X_\ell =0\} = \delta$ could not hold. This proves \eqn{eq-vacancy-2}.}

From property  \eqn{eq-vacancy-2} we have
$$
\lim_c \pr\{X_n(\infty) =0, ~n \le \ell \} = \delta, ~~~~\forall \ell\ge 1,
$$
which implies that for an arbitrarily large fixed $\ell$, the probability that sites $1,\ldots,\ell$ are simultaneously completely occupied remains bounded away from $0$. This is, however, impossible, because: \edit{$\cl(\alpha,c)$ is dominated by $\Psi_{c-1} \cl(\alpha,1)$; $\cl(\alpha,1)=\nu_\alpha$ (because $\alpha \le 1/2$); and then under $\Psi_{c-1} \cl(\alpha,1)$ the probability 
of sites $1,\ldots,\ell$ being simultaneously occupied can be made arbitrarily small by making $\ell$ large.} The contradiction completes the proof.
 $\Box$

 {\em Proof of Theorem~\ref{th-main-res}(i).} Let us show that, for a fixed $c\ge 2$, $\alpha_c^* < \alpha_{c-1}^*$.
 Consider $\alpha<  \alpha_{c-1}^*$. We will show that $\phi(\alpha,c)=1/4$, when 
 $\alpha$ is close to $\alpha_{c-1}^*$.
 
 Consider processes for $(\alpha,c-1)$- and $(\alpha,c)$-systems, coupled as specified in Section~\ref{sec-c-coupling}.
 We can and do augment this coupling by the coupling of $(\alpha,c-1)$- and $(\alpha_{c-1}^*,c-1)$-systems,
 specified in Section~\ref{sec-alpha-coupling} (with $c-1$ replacing $c$, and $\alpha_{c-1}^*$ replacing $\hat\alpha$).
 Let us use the following terminology. We will call the class-1 particles, common for both couplings, ``black;'' the class-2 particles in the coupling of $(\alpha,c-1)$- and $(\alpha,c)$-systems, will be called ``blue;'' and the class-2 particles in the coupling of $(\alpha,c-1)$- and $(\alpha_{c-1}^*,c-1)$-systems, will be called ``red.'' The evolutions of the processes of black, blue and red particles are all coupled, but blue and red particles are oblivious of each others' existence. (The evolutions of blue and red particles are, of course, {\em not} independent.) Specifically, at each common arrival bell ring or an $(n,n+1)$-bell ring, the algorithms for the black-and-blue system (as in Section~\ref{sec-c-coupling})
 and for the black-and-red system (as in Section~\ref{sec-alpha-coupling}) are performed separately; obviously, both algorithms will result 
 in the same movement of black particles, and therefore there is a single common black particles' process. 
 
  Let us further augment our coupling construction. Consider ``purple'' particles, which are ``derived'' from blue particles by ``impeding'' the movement 
 of the latter as follows. When a blue particle is added to site $1$, a new purple particle is placed into an artificial infinite capacity First-In-First-Out
  queue at the system entrance,
 say at an artificial site $0$. A purple particle from this entrance queue can actually enter the system at site $1$ 
 only ``within'' a red particles entering the system at site $1$. (Note that a purple particle can never be placed into the entrance queue at the same time as a red particle enters the system.) Within the system, assume that red particles containing purple particles
 have priority over red particles not containing purple. (Formally, we replace black-and-red system by the system with three priority classes -- black, red-with-purple, and red-without-purple, with the latter being ``class 3".) 
  
 Suppose, the system starts from the empty state.
 Observe the following: 
 all purple particles at sites $n \ge 1$ always stay within red particles; 
 the combined black-and-purple system at sites $n \ge 1$ complies with the back-pressure rule; 
 then, by Lemma~\ref{lem-ahead}, w.p.1, the combined black-and-purple process ``stays behind'' the combined black-and-blue process
 in the sense of order $\preceq$ at all times; since the black process is common, we also have that purple process stays behind blues process
 in the sense of $\preceq$.

 Note that the
 steady-state flux due to red particles is exactly  $1/4-\phi(\alpha,c-1)$; this is because it is equal to the
 additional (to black particles)
 steady-state flux they contribute in black-and-red system, and the latter system has flux  $\phi(\alpha_{c-1}^*,c-1)=1/4$ by definition.
 Denote by $r_{red,in}$ and $r_{blue,in}$ the steady-state rates at which red and blue particles actually get into the system.
 Obviously, $r_{red,in} \le \alpha_{c-1}^*-\alpha$. By Lemma~\ref{lem-rate-to-floor-c}(i), we can choose $\alpha$ sufficiently close to $\alpha_{c-1}^*$,
 so that 
 \beql{eq-in-rates}
 r_{red,in} < s(\alpha,c) \le r_{blue,in}.
 \eeql
 
 Let us go back to the process starting from the empty state. Given \eqn{eq-in-rates}, we see that, w.p.1, as $t\to\infty$
 the entrance queue of purple particles increases to infinity and, therefore, the total number of red particles not containing purple particles,
 that will ever exist, is finite. This means that the long-run average flux of purple particles (equal to their steady-state flux under the LIM)
 is equal to that of red particles, which is $1/4-\phi(\alpha,c-1)$. Then the long-run average flux of combined black and purple particles is $1/4$.
 Then $\phi(\alpha,c)$, which is the long-run average flux of combined black and blue particles, is at least $1/4$, and then equal to $1/4$.
 This completes the proof of $\alpha_c^* < \alpha_{c-1}^*$.
 
  Let us show that $\alpha_c^* \downarrow 1/4$. Suppose not, and $\alpha_c^* \downarrow \alpha > 1/4$. Then, by
 Lemma~\ref{lem-rate-to-floor-c}(ii), we could choose $c$ large enough so that,
 for the coupling construction above in the proof, 
  $r_{red,in} \le \alpha_{c-1}^*-\alpha < r_{blue,in} =r_{2,in}(\alpha,c)$.
 Therefore, by the same argument as
 in the proof above,  $\phi(\alpha,c)=1/4$ -- a contradiction.
  $\Box$
  
  \subsection{Proof of Theorem~\ref{th-main-res}(ii)}
  \label{sec-main-proof-ii}
    
\edit{
Fix any $\alpha < \alpha_{c-1}^*$. 
Let us show that the effective floor of $\cl(\alpha,c)$ is $1$.
It suffices to show this for $\alpha \in (\alpha_c^*,\alpha_{c-1}^*)$.
Suppose not, that is the effective floor of $\cl(\alpha,c)$ is at least $2$;
this means that floor $1$ is completely occupied. 
Then $\cl(\alpha,c) = \Psi_1 \cl(\alpha,c-1)$, which implies $\phi(\alpha,c) = \phi(\alpha,c-1)$.
But this is impossible because $\phi(\alpha,c)=1/4$ and $\phi(\alpha,c-1)<1/4$.
 }
  
  \edit{
  Let us show that $\cl(\alpha,c)=\Psi_1 \cl(\alpha,c-1)$ for $\alpha \ge \alpha_{c-1}^*$. These two measures have the same flux $1/4$, and 
  $\Psi_1 \cl(\alpha,c-1)$ dominates $\cl(\alpha,c)$. Then they are equal by Lemma~\ref{lem-strict-dom}(ii). $\Box$
  }

  \subsection{Proof of Lemma~\ref{lem-phi-props}}
  \label{sec-phi-props-proof}
  
   Fix $c$. We already know that $\phi(\alpha,c)$ is non-decreasing Lipschitz continuous in 
  $\alpha$ in $[0,\alpha^*_c]$. Fix any $\alpha < \alpha^*_c$. By definitions, $\phi(\alpha,c) < 1/4$ and $\phi(\alpha^*_c,c) = 1/4$.
  Fix any $\hat \alpha \in (\alpha,\alpha^*_c)$. Consider a system with three particle classes, $1$, $2$, and $3$, with decreasing priorities.
  Class-1 particles have the attempted exogenous arrival rate $\alpha$ at site $1$; class-2 and class-3 particles' attempted arrival rates
  are $\hat \alpha-\alpha$ and $\alpha^*_c - \hat \alpha$, respectively. Given the priority order, in steady-state an arriving class-2 particle has at
  least as high probability of being accepted and never leaving the system than a class-3 particle. Therefore,
  $$
  \frac{\phi(\hat \alpha,c)-\phi(\alpha,c)}{\hat \alpha - \alpha} \ge \frac{\phi(\alpha^*_c,c)-\phi(\alpha,c)}{\alpha^*_c - \alpha} > 0,
  $$
  and then $\phi(\hat \alpha,c) > \phi(\alpha,c)$. Then, same argument shows the concavity, if we consider any three attempted arrival rates,
  $\alpha < \hat \alpha < \alpha'$, with $\alpha' \le \alpha^*_c$.
  
  For a fixed $\alpha>0$, let us show that $\phi(\alpha,c)<1/4$ implies $\phi(\alpha,c) < \phi(\alpha,c+1)$. This is equivalent to showing that,
   in the system with $c+1$ floors, 
   the flux of 
  $\Psi_1 \cl(\alpha,c)$,
  which has effective floor $2$, 
  is strictly smaller than the flux of $\cl(\alpha,c+1)$, whose effective floor is $1$. \edit{This is true by Lemma~\ref{lem-strict-dom}(i),
  because $\Psi_1 \cl(\alpha,c)$ dominates $\cl(\alpha,c+1)$ and they are not equal.}
    
  For a fixed $\alpha >1/4$, we immediately see that $\phi(\alpha,c)$ is strictly increasing in $c$ up to the point $c=c_\alpha$.
  
  Finally, consider the dependence of $\phi(\alpha,c)$ on $c$ when $\alpha \le 1/4$ is fixed. Clearly, $\phi(\alpha,c) < \alpha \le 1/4$ for all $c$,
  because the probability of an arriving particle being blocked is always positive. Then, it suffices to show that, for a fixed $\alpha<1/4$,
  $\phi(\alpha,c) \uparrow \alpha$ as $c \uparrow \infty$. It is shown in  \cite{St2011-tandem-queues} that, when $\alpha<1/4$, the system,
  where there is no limit on the number of particles at any site, has a proper LIM -- let us denote it $\cl(\alpha,\infty)$ -- with the number of particles at 
  any site being finite a.s. 
  By monotonicity, $\cl(\alpha,\infty)$  dominates $\cl(\alpha,c)$
  for any $c$. This means that the steady-state probability of site $1$ being completely occupied (having $c$ particles) vanishes as $c\to\infty$.
  This implies that the blocking probability vanishes, and then $\phi(\alpha,c) \uparrow \alpha$.
 $\Box$

  \section{Left- and right-side asymptotic limits (special case)} 
  \label{sec-left-right}
  
    In Section~\ref{sec-model-one-side} we analyzed the LIM $\cl(\alpha,c)$ of the one-sided system; this LIM is equal to the left-side limit, as $N\to\infty$ of stationary distributions of the system with $\alpha\le 1$ and $\beta = 1$ (in fact, $\beta \ge 1$). In this section
we will prove Theorem~\ref{th-special-max-flux}, which gives both the left- and right-side stationary distribution limits 
in the case when $1/2 \le \beta \le 1$ and the maximum possible limiting flux $1/4$ is achieved.
In particular, the theorem shows that if $\alpha \ge \alpha^*_c$ (i.e., $c \ge c_\alpha$),
the left-side asymptotic limit for any $\beta\ge 1/2$ is exactly same as for $\beta=1$,
i.e. equal to $\cl(\alpha,c)$.
    
We need more notation. 
Recall that the right-side asymptotics is such that the sites are (re-)labeled from right to left, so that, under this labeling, particles move ``left'' and leave system from site $1$, while holes move ``right'' at enter the system at site $1$. For a distribution $\ch$, we will denote by $\ch^{\updownarrow}$ the same distribution,
 but with particles and holes switched. So, for example, $[\cl(\beta,c)]^{\updownarrow}$ is the LIM (lower - in terms of holes) for one-sided system with $c$ floors and with holes entering (at site $1$) at the attempted rate $\beta$.
 
 Before stating Theorem~\ref{th-special-max-flux}, we recall that for the single-floor system the left- and right-side stationary distribution limits are well known,
 as summarized in the following
 
  \begin{prop}[Corollary from corollary 3.17 in \cite{Liggett-1975}] 
 \label{prop-single-floor}
 Consider a finite single-floor system (which is a standard TASEP).
 WLOG, let $\alpha \le \beta$.  Then, as $N\to\infty$:\\
   if $\alpha \le 1/2$, the left-side asymptotic limit of stationary distribution is $\nu_\alpha$ and the right-side asymptotic limit of stationary distribution behaves like $\nu_\alpha$ at $\infty$; \\
 if $\alpha > 1/2$, both the left-side and right-side asymptotic limits of stationary distribution behave like $\nu_{1/2}$ at $\infty$.
 \end{prop}

  \begin{thm}
  \label{th-special-max-flux}
  Consider a finite system with the number of floors $c\ge 1$ and $1/2 \le \beta \le 1$. Suppose $\alpha \ge \alpha^*_c$, i.e. $c \ge c_\alpha$, 
  and let $\ell = c-c_\alpha$. 
   Then, as $N\to\infty$, 
  the steady-state flux converges to $\phi(\alpha,c)=1/4$ and the left-side asymptotic limit $\ch_l$ of the stationary distribution is $\cl(\alpha,c)=
  \Psi_\ell \cl(\alpha,c-\ell)$.
  In addition,
  the right-side asymptotic limit $\ch_r$ of the stationary distribution (from the point of view of holes) is
  $\ch_r=[\cl(\beta,c)]^{\updownarrow}=[\Psi_{c-1} \cl(\beta,1)]^{\updownarrow}$
  (which, by Proposition~\ref{prop-single-floor},
  means that $\ch_r$ has flux $1/4$ and behaves like $[\Psi_{c-1} \nu_{1/2}]^{\updownarrow}$ at $\infty$.)
  \end{thm}

For the proof of Theorem~\ref{th-special-max-flux}, we will need the following

\begin{lem}
\label{lem-1floor-dyn}
Consider one-sided system with $c=1$ (single-floor TASEP), with $\alpha \in [1/2,1]$. Then:

(i) $\cl(\alpha,1)$  stochastically dominates $\nu_{1/2}$, and behaves like $\nu_{1/2}$ at $\infty$ (and, in particular, $\phi(\alpha,1)=1/4$). [This is a corollary from \cite[theorems 1.8(a) and 1.7]{Liggett-1975}.]

(ii) Suppose, the initial distribution at time $0$ is $\cl(1,1)$. Then, starting this initial state, the process $Q(t)$ is stochastically monotone non-increasing in time $t\ge 0$.  As $t\to\infty$ the distribution of $Q(t)$ converges to $\cl(\alpha,1)$. Moreover, the instantaneous flux into site $1$ 
(namely, $\alpha \times \pr\{Q_1(t)=0\}$) is non-decreasing in time, converging to $1/4$ as $t\to\infty$. In particular, for any $\epsilon>0$ we can choose $T=T(\epsilon)>0$ sufficiently large so that the instantaneous flux at times $t\ge T$ is at least $1/4-\epsilon$.
\end{lem}

  {\em Proof of Lemma~\ref{lem-1floor-dyn}(ii).} The monotonicity of the distribution of $Q(t)$ is obvious, because $\alpha \le 1$ and $\cl(1,1)$ is stationary for $\alpha=1$.
  To show convergence of the distribution to $\cl(\alpha,1)$, note that  at any time $t\ge 0$,
the distribution of $Q(t)$ dominates $\cl(\alpha,1)$ and is dominated by $\cl(1,1)$. Both $\cl(\alpha,1)$ and $\cl(1,1)$ behave like $\nu_{1/2}$ at $\infty$.
\edit{Then, the  limit $\tilde \ch$ of the distribution of $Q(t)$ is stationary, dominates $\cl(\alpha,1)$ and has the same flux $1/4$.
By Lemma~\ref{lem-strict-dom}(ii), $\tilde \ch =\cl(\alpha,1)$.}

The instantaneous flux $\alpha \times \pr\{Q_1(t)=0\}$ is non-decreasing because $\pr\{Q_1(t)=0\}$. And it must converge to $1/4$, which is the flux of the limiting distribution $\cl(\alpha,1)$. 
  $\Box$

  {\em Proof of Theorem~\ref{th-special-max-flux}.} Consider first the case $\ell=0$.
  
    We start with the special case $\beta=1$. In this case,  by Theorem~\ref{th-main-res}, $\ch_l=\cl(\alpha,c)$, which has effective floor $1$ and flux $\phi(\alpha,c)=1/4$.
  Consider any distribution 
  $\ch_r$, which is a subsequential limit under the right-side asymptotics. From the point of view of holes, $\ch_r$ is a stationary distribution of the one-sided process, with holes entering at attempted rate $\beta=1$, and it has flux $1/4$; 
  $\ch_r$ has effective floor $c$ (from the point of view of holes), simply because $\beta=1 \ge 1/2$ (and $\alpha\le 1$).
Let $\hat \ch_r$ be defined via $\ch_r = \Psi_{c-1} \hat \ch_r$.
Then, 
  $\hat \ch_r$  is a stationary distribution of the one-sided single-floor process (standard TASEP), with holes entering at attempted rate $\beta=1$,
  and with flux $1/4$.
  Then, from the point of view of holes, $\hat \ch_r$ stochastically dominates $[\cl(1,1)]^{\updownarrow}$, which, recall, also has flux $1/4$. 
  \edit{By Lemma~\ref{lem-strict-dom}(ii), $\hat \ch_r = [\cl(1,1)]^{\updownarrow}$.}

Now, consider the more general case of $\beta \in [1/2,1]$. 
Let us first find the  right-side asymptotic limit $\ch_r$.
Suppose the initial distribution of the process at time $0$ is the stationary distribution of the system with parameter $\beta=1$.
Now, consider the right-side asymptotics, from the point of view of holes.
We know from the paragraph just above, that from the point of view of holes, the initial distribution converges to $[\Psi_{c-1} \cl(1,1)]^{\updownarrow}$ as $N\to\infty$.
Given the specified initial distribution in a system with finite $N$,
 from the point of view of holes,
the distribution will be stochastically non-increasing in time $t$, while always dominating the stationary distribution,
and with the instantaneous flux at site $1$  (under right-side asymptotics) non-decreasing in time. 
Fix an arbitrarily small $\epsilon>0$ and then the corresponding 
$T=T(\epsilon)$ as in Lemma~\ref{lem-1floor-dyn}(ii) with $\alpha=\beta$. 
Then, for any fixed $M>0$, for all sufficiently large $N$, the evolution of the process at sites $1,\ldots,M$ (under the right-side asymptotics) on floor $c$, 
if we view it as a single-floor process, is close to the evolution of the single-floor TASEP, described in Lemma~\ref{lem-1floor-dyn}(ii) (with $\alpha=\beta$).
We conclude that, for our original process, for all sufficiently large $N$, the instantaneous flux at site $1$ (under right-side asymptotics) at times $t\ge T$ is at least $1/4 -2\epsilon$. Therefore, the limiting steady-state flux is at least $1/4$, and then equal to $1/4$.
Consider now any distribution 
  $\ch_r$, which is a subsequential limit under the right-side asymptotics. 
  We see that the flux of $\ch_r$ is $1/4$.  We also see that $\ch_r$ has to be dominated (from the point of view of holes) by $[\Psi_{c-1} \cl(\beta,1)]^{\updownarrow}$;
  indeed, $\ch_r$ is a distributional limit of the process, which is monotone non-increasing and its evolution on any finite time interval
  converges to the evolution of the process in Lemma~\ref{lem-1floor-dyn}(ii) (with $\alpha=\beta$), ``lifted up'' to floor $c$.
  \edit{By Lemma~\ref{lem-strict-dom}(ii), 
   $\ch_r = [\Psi_{c-1} \cl(\beta,1)]^{\updownarrow}$.} 
  
  Finally, the left-side asymptotic limit $\ch_l$ must be equal to $\cl(\alpha,c)$, \edit{again by Lemma~\ref{lem-strict-dom}(ii), because the former dominates the latter and must have the same flux $1/4$.}
  
  Generalization to arbitrary $\ell \ge 0$ is as follows. Increasing the number of floors from $c-\ell$ to $c$ 
  cannot reduce flux, so that the limiting flux is still $1/4$. 
  Then the right-side asymptotic limit $\ch_r$ is obtained in exactly same way; 
  so, $\ch_r = [\Psi_{c-1} \cl(\beta,1)]^{\updownarrow}$. The left-side asymptotic limit $\ch_l$ has flux $1/4$ and is dominated by  $\Psi_{\ell} \cl(\alpha,c-\ell)$,
  \edit{which also has flux $1/4$; 
  by Lemma~\ref{lem-strict-dom}(ii), $\ch_l = \Psi_{\ell} \cl(\alpha,c-\ell)$.}
  $\Box$
  
 We now introduce terminology which will help interpret the simulation results is Section~\ref{sec-simulations} and state 
 the general Conjecture~\ref{conj-obbc}  
 about the system asymptotic limit in Section~\ref{sec-obbc}.
 
 Recall that for a given state $Q$ and $m\le c$ we defined
$$
A_{m} = \sup\{n\ge 1 ~|~ Q_k \ge m, ~\forall k \le n\},
$$
where, by convention, $A_m=0$ if the set under $\sup$ is empty; note that $A_0 = N$. Then, sites $\{1, \ldots, A_{c-1}\}$ form a ``floor-$c$ zone,'' where there is either $c$ or $c-1$ particles in each site. (If $A_{c-1} < 1$, this zone is empty.) Similarly, for $m=c-1, \ldots, 1$, in view of \eqn{eq-queue-cond},
sites $\{A_{m} + 1, \ldots, A_{m-1}\}$ form a ``floor-$m$ zone,'' where there is either $m$ or $m-1$ particles in each site. 
(If $A_{m} = A_{m-1}$, floor-$m$ zone is empty.) These zones are random, of course, depending on the current state $Q(t)$.
We will say that a floor-$m$ zone is {\em large} (asymptotically, as $N\to\infty$) if the size of this zone in steady-state remains $O(N)$;
zones whose size is $o(N)$ are {\em vanishing}. We will say that a floor-$m$ zone is {\em large stable} (asymptotically, as $N\to\infty$),
if its size normalized by $N$ converges to a positive constant and, moreover, the position of the zone within the space dimension normalized 
to interval $[0,1]$, converges to a fixed one.

In a single-floor (standard TASEP) system, there is, of course, only one, fixed zone on floor $1$, occupying the entire space. 
In addition to Proposition~\ref{prop-single-floor}, it is known (cf. \cite[section 2.2]{Blythe_2007})
that in the single-floor system, when $\alpha \le 1/2$ and $\alpha < \beta$, as $N\to\infty$, the bulk density of the particles in steady-state converges to constant density $\alpha$ and, moreover, ``within the floor-$1$ zone'' (away from the right boundary) the stationary distribution becomes 
$\nu_\alpha$. When $\alpha \ge 1/2$ (and then $\beta \ge 1/2$), the bulk density is $1/2$ and the stationary distribution within the floor-$1$ zone
(away from both boundaries) becomes $\nu_{1/2}$. For example, when $\alpha < \beta \le 1/2$, the system behavior can be interpreted as follows: kinematic wave of particles at density $\alpha$ propagate from the left end to the right, kinematic wave of holes at density $\beta$ propagate 
from the right end to the left; after these two waves meet and form a shock (discontinuity of the density), 
the shock point will move right (given $\alpha < \beta$) until
the density $\alpha$ ``occupies'' the bulk of the zone; note that in this case the flux $\alpha(1-\alpha)$ of the particles moving from the left is smaller
than the flux $\beta(1-\beta)$ of holes moving from the right; so, the system limiting flux is the smaller of the two.
By interchangeability of particles and holes, in the case $\beta < \alpha \le 1/2$,
the bulk of the zone will have holes' density $\beta$ (and then particle density $1-\beta$), and the system limiting flux is $\beta(1-\beta)$.
This behavior of the single-floor process provides a key intuition for our Conjecture~\ref{conj-obbc}.

 \section{Simulations}
  \label{sec-simulations}
  
  We now present some simulations to illustrate our formal results and motivate a general Conjecture~\ref{conj-obbc} in Section~\ref{sec-obbc}.
   Each simulation is of the system with $N=1200$ sites, the simulation time is 
  $2\cdot 10^9 / N$, the initial state is ``empty'' (all queues are zero). In the figures: the $x$-axis is the space dimension with site index increasing from left to right; the $y$-axis is the time dimension, with the time running ``downward''; a site color at a given time indicates its queue length, with lighter colors corresponding to higher queue lengths.
  
 Consider first a simple case of $c=2$, $\alpha=\beta=1$. By Theorem~\ref{th-main-res}, 
 the left-side steady-state limit is $\Psi_1 \cl(1,1)$ with effective floor $2$ 
  and flux $1/4$. By symmetry, the right-side steady-state limit is same, but from the point of view of holes. 
  Consistent with that, we see (Figure~\ref{multi_tasep_100_100-c2-zero-init-dur2000}) 
  what appears to be two large stable zones, on floor $2$ on the left and floor $1$ on the right, 
  of roughly same size, with the front (boundary) between them being roughly in the middle.
 In the case $c=3$ and $\alpha=\beta=1$ we see (Figure~\ref{multi_tasep_100_100-c3-zero-init-dur2000}) three large stable, floor-$1$, -$2$, and -$3$, zones, 
  with roughly equal size. 
  
To illustrate the phase transition of the flux with respect to $\alpha$, consider, first, the process with $\alpha=0.45$, $\beta=1$, with two values of $c=2,3$. (See simulations in Figures~\ref{multi_tasep_45_100-c2-zero-init-dur2000}-\ref{multi_tasep_45_100-c3-zero-init-dur2000}.) 
We see that the effective floor of $\cl(0.45,2)$ is $1$, which it has to be according to Theorem~\ref{th-main-res}, 
because $0.45 < \alpha^*_1 = 1/2$; which means that there is only one large stable zone on floor $1$.
We also see that when $c$ increases from $2$ to $3$, the left-side limit effective floor appears to change from $1$ to $2$
(and we see two large stable zones, of about equal size, on floors $2$ and $1$).
This means that
  $0.45 \ge \alpha^*_2$, and therefore flux $\phi(0.45,2)=\phi(0.45,3)=1/4$; because if not,
  then according to Theorem~\ref{th-main-res} the effective floor of $\cl(0.45,3)$ would be $1$. 
  
  To ``test'' Theorem~\ref{th-special-max-flux}, consider the case [$\alpha=0.45$, $\beta=0.6$, $c=2$] 
  (see Figure~\ref{multi_tasep_45_60-c2-zero-init-dur2000})
  in which, compared to [$\alpha=0.45$, $\beta=1$, $c=2$], parameter $\beta$ is reduced to $0.6$. 
  In both cases we see what appears to be the same left-side limit, with effective floor $1$ and flux $1/4$. 
  
  Consider now the process with $\alpha=0.3$, $\beta=1$, and with $c=2,3,4$. (See simulations in Figures~\ref{multi_tasep_30_100-c2-zero-init-dur2000}-\ref{multi_tasep_30_100-c4-zero-init-dur2000}.) When $c=2$, the effective floor of the left-side limit appears to be $1$, and therefore 
  there is only one large stable zone, on floor $1$. The particle density in the floor $1$ zone is about $0.46 < 1/2$ and the corresponding flux is approximately
  $0.46(1-0.46) < 1/4$. When $c=3$, the effective floor of the left-side limit appears to be still $1$, but appears to change to $2$ when $c$ increases to $4$
  (and we see two large stable zones of approximately equal size); 
then, by Theorem~\ref{th-main-res}, $\alpha^*_3 \le 0.3 < \alpha^*_2$, and therefore flux $\phi(0.3,2) < \phi(0.3,3)=\phi(0.3,4)=1/4$. 

Consider now the system with $\alpha=0.3$, $\beta=0.48$, $c=2$, which is outside of the scope of Theorem~\ref{th-special-max-flux} since $\beta< 1/2$;
see simulation in Figure~\ref{multi_tasep_30_48-c2-zero-init-dur2000}. 
Our simulation of the case $\alpha=0.3$, $\beta=1$, $c=2$, described above, suggests that the flux is about $0.46(1-0.46)$, which means that the 
bulk density parameter $\gamma$ in the floor-$1$ zone is about $0.46 < \beta = 0.48$. We observe that
the stationary distribution in the case [$\alpha=0.3$, $\beta=0.48$, $c=2$] is ``almost same'' as in the case [$\alpha=0.3$, $\beta=1$, $c=2$],
in the sense of having the single large stable  floor-$1$ zone, with the same limiting density, and same limiting flux. 

Next, consider the case [$\alpha=0.3$, $\beta=0.4$, $c=2$]; see simulation in Figure~\ref{multi_tasep_30_40-c2-zero-init-dur2000}.
Here, compared to the cases [$\alpha=0.3$, $\beta=1$, $c=2$] and [$\alpha=0.3$, $\beta=0.48$, $c=2$], we make $\beta < \gamma$,
where $\gamma \approx 0.46$ is the density in the single large stable floor-$1$ zone, arising in those cases. 
We observe that the effective floor of the left-side limit appears to be still $1$, and therefore floor-$1$ is still the only large stable zone.
However, the bulk density we observe in the floor-$1$ zone is about $0.6$, which means that the holes' density appears to be equal to $\beta=0.4$.
Consequently, the flux has decreased from $\gamma(1-\gamma) \approx 0.46(1-0.46)$ to a smaller value $\beta(1-\beta) = 0.4(1-0.4)$.

Let us now go from the case [$\alpha=0.3$, $\beta=0.4$, $c=2$] to case [$\alpha=0.3$, $\beta=0.4$, $c=3$], 
i.e. increase the number of floors from $2$ to $3$;
see simulation in Figure~\ref{multi_tasep_30_40-c3-zero-init-dur2000}. 
We see, again, the single large stable zone, but it is of completely different type. It is now the floor-$2$ zone.
(So, the left-side limit has effective floor $2$, and the right-side limit, from the point of view of holes,
also has effective floor $2$.) The particle density in the floor-$2$ zone is about $1.46$, so ``on floor $2$'' it is about $0.46$,
and the flux is about $0.46(1-0.46)$. We see that the flux is greater than in the case $c=2$ (as it should be), 
but is still less than the maximum possible limiting flux $1/4$. Let us increase the number of floors to $4$ -- see simulation
for the case [$\alpha=0.3$, $\beta=0.4$, $c=4$] in Figure~\ref{multi_tasep_30_40-c4-zero-init-dur2000}. Again, the only large stable zone in on floor $2$,
but the density on this floor is now about $1/2$, and therefore the maximum limiting flux $1/4$ is achieved.
Increasing the number of floors to $5$ (simulation of case  [$\alpha=0.3$, $\beta=0.4$, $c=5$] in Figure~\ref{multi_tasep_30_40-c5-zero-init-dur2000}),
we see that now we have two large stable zones, of about equal size, on floors $3$ and $2$.

      \begin{figure}
\centering
        \includegraphics[width=6in]{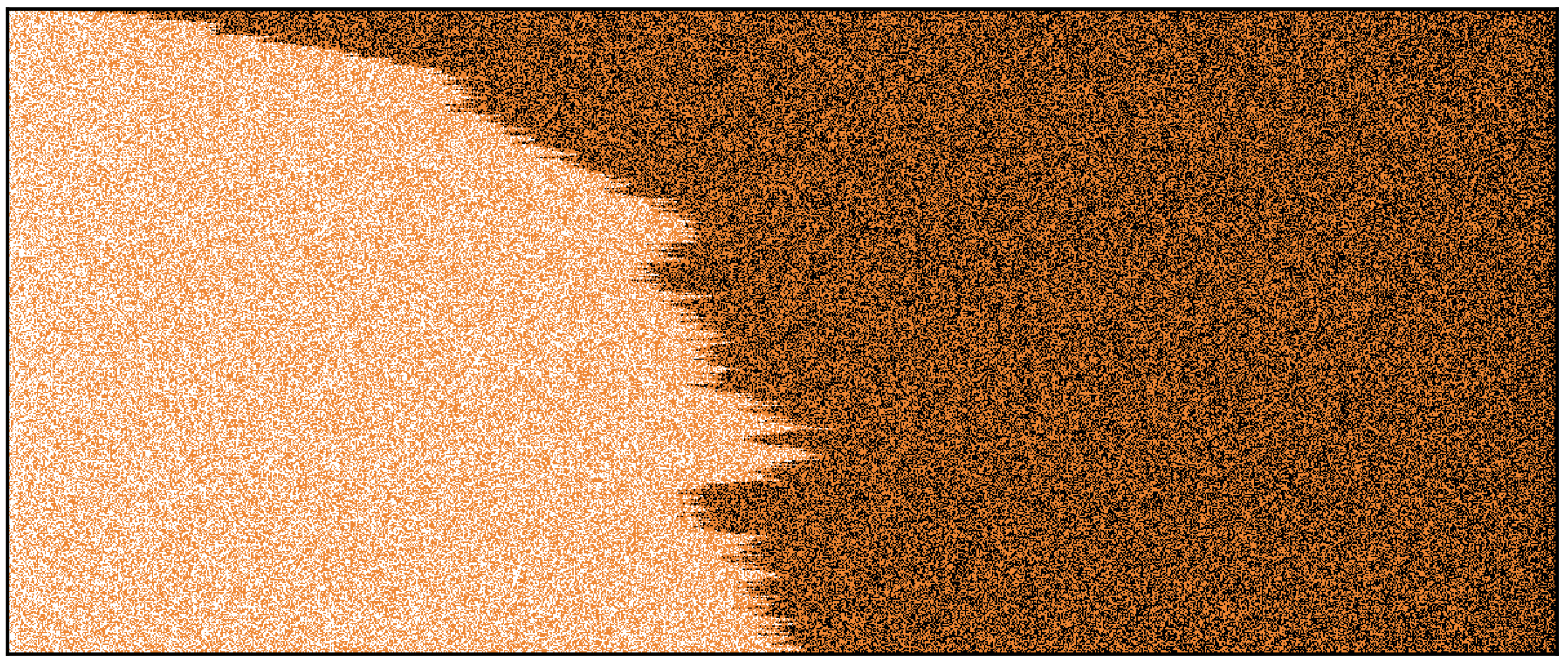}
    \caption{$\alpha=1, \beta=1, c=2$}
    \label{multi_tasep_100_100-c2-zero-init-dur2000}
\end{figure}

      \begin{figure}
\centering
        \includegraphics[width=6in]{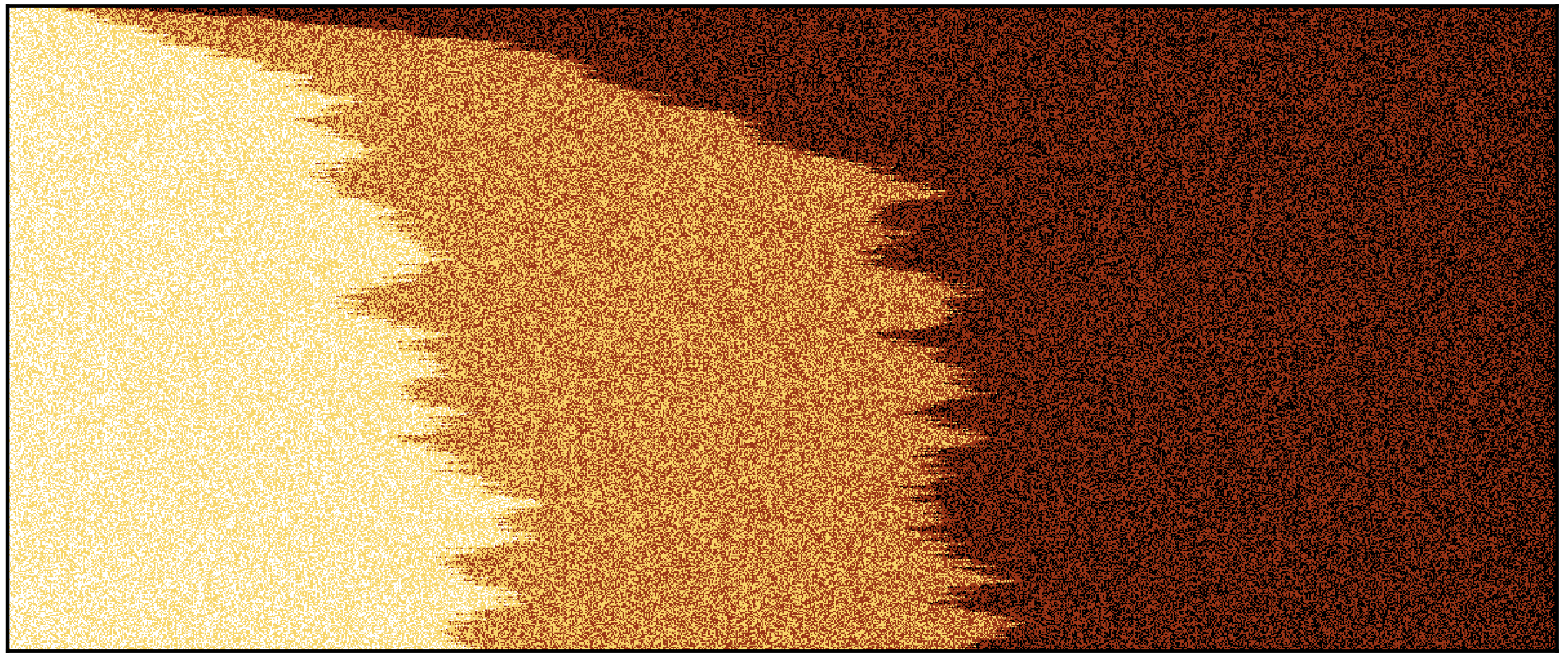}
    \caption{$\alpha=1, \beta=1, c=3$}
    \label{multi_tasep_100_100-c3-zero-init-dur2000}
\end{figure}

    \begin{figure}
\centering
        \includegraphics[width=6in]{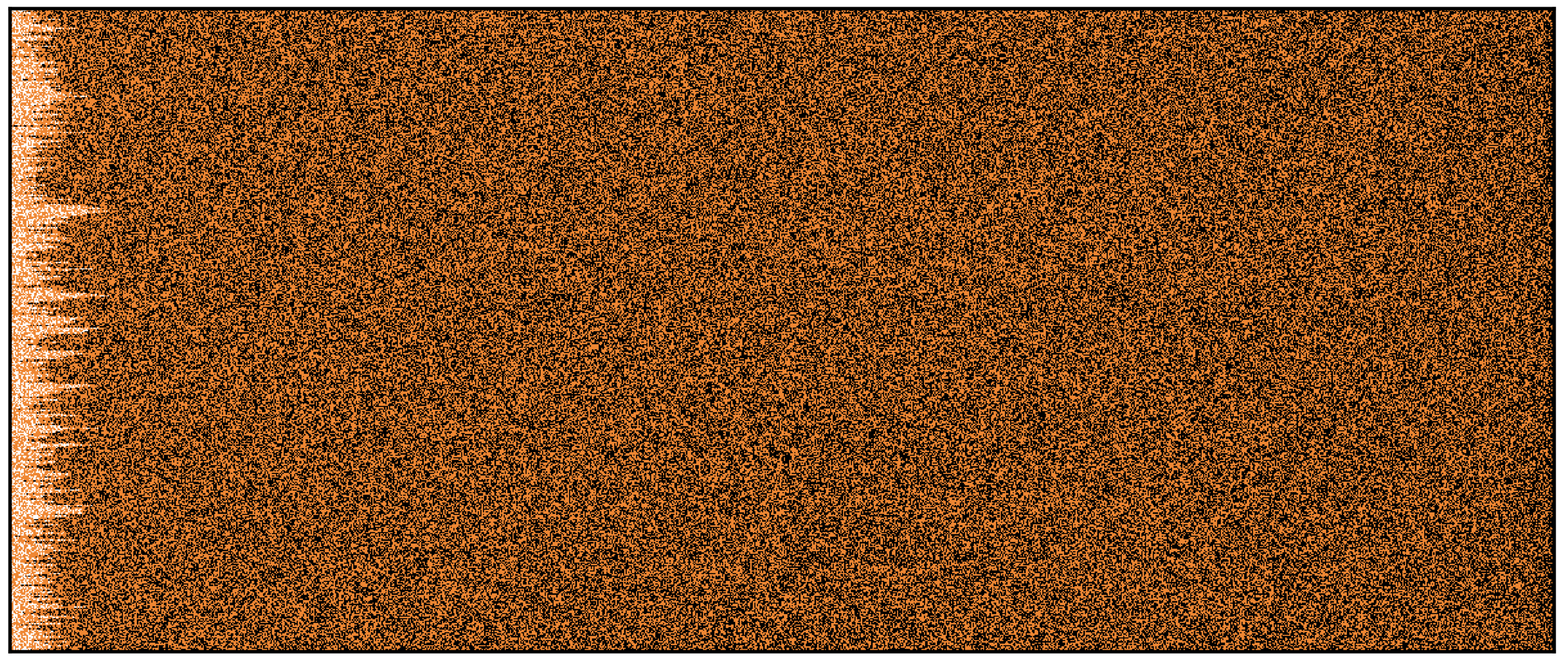}
    \caption{$\alpha=0.45, \beta=1, c=2$}
    \label{multi_tasep_45_100-c2-zero-init-dur2000}
\end{figure}

  \begin{figure}
\centering
        \includegraphics[width=6in]{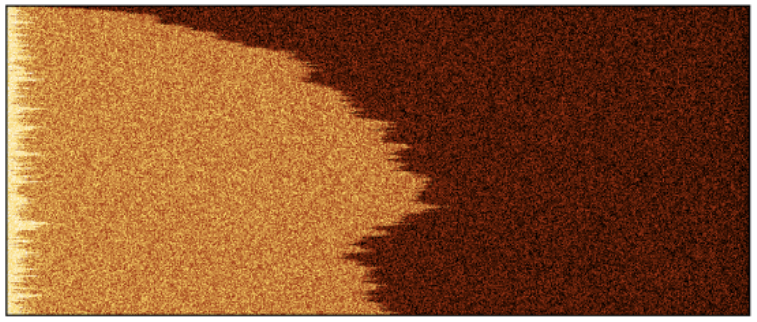}
    \caption{$\alpha=0.45, \beta=1, c=3$}
    \label{multi_tasep_45_100-c3-zero-init-dur2000}
\end{figure}

  \begin{figure}
\centering
        \includegraphics[width=6in]{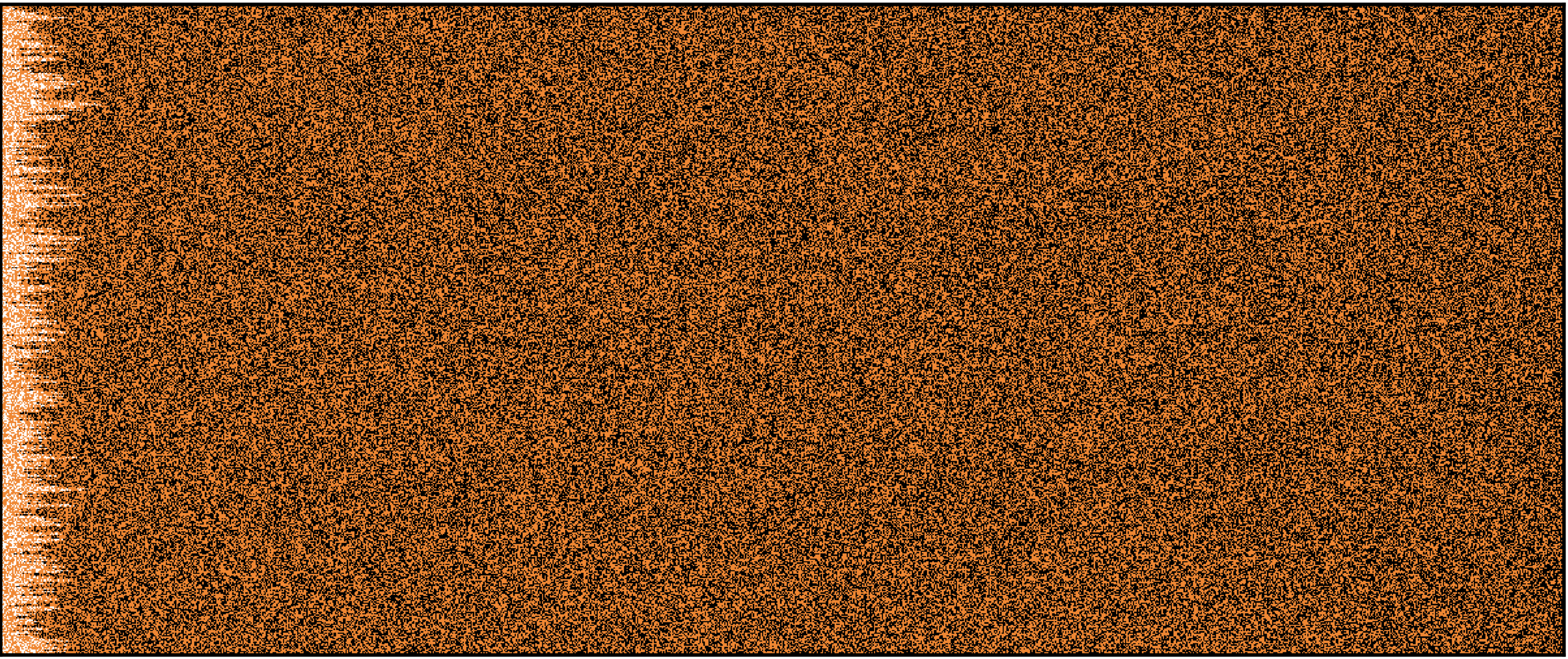}
    \caption{$\alpha=0.45, \beta=0.6, c=2$}
    \label{multi_tasep_45_60-c2-zero-init-dur2000}
\end{figure}

  \begin{figure}
\centering
        \includegraphics[width=6in]{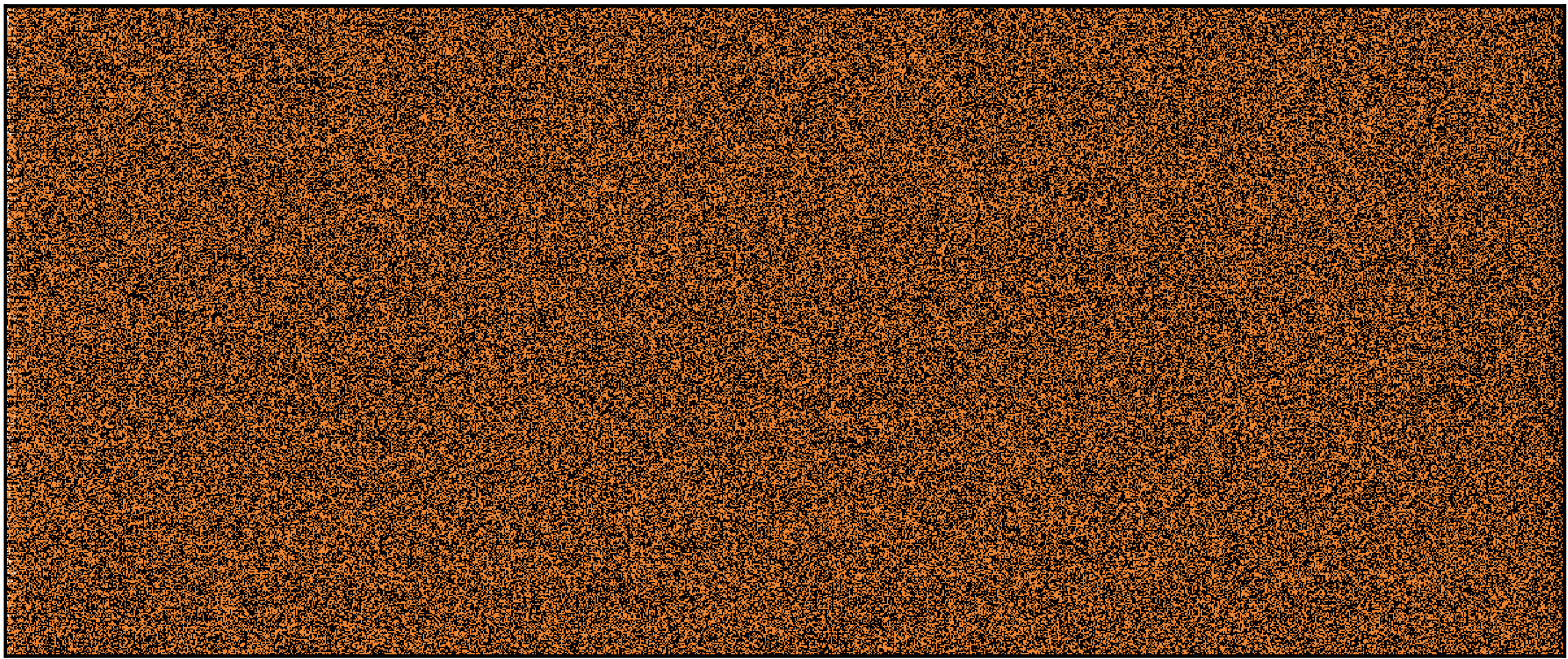}
    \caption{$\alpha=0.3, \beta=1, c=2$}
    \label{multi_tasep_30_100-c2-zero-init-dur2000}
\end{figure}

  \begin{figure}
\centering
        \includegraphics[width=6in]{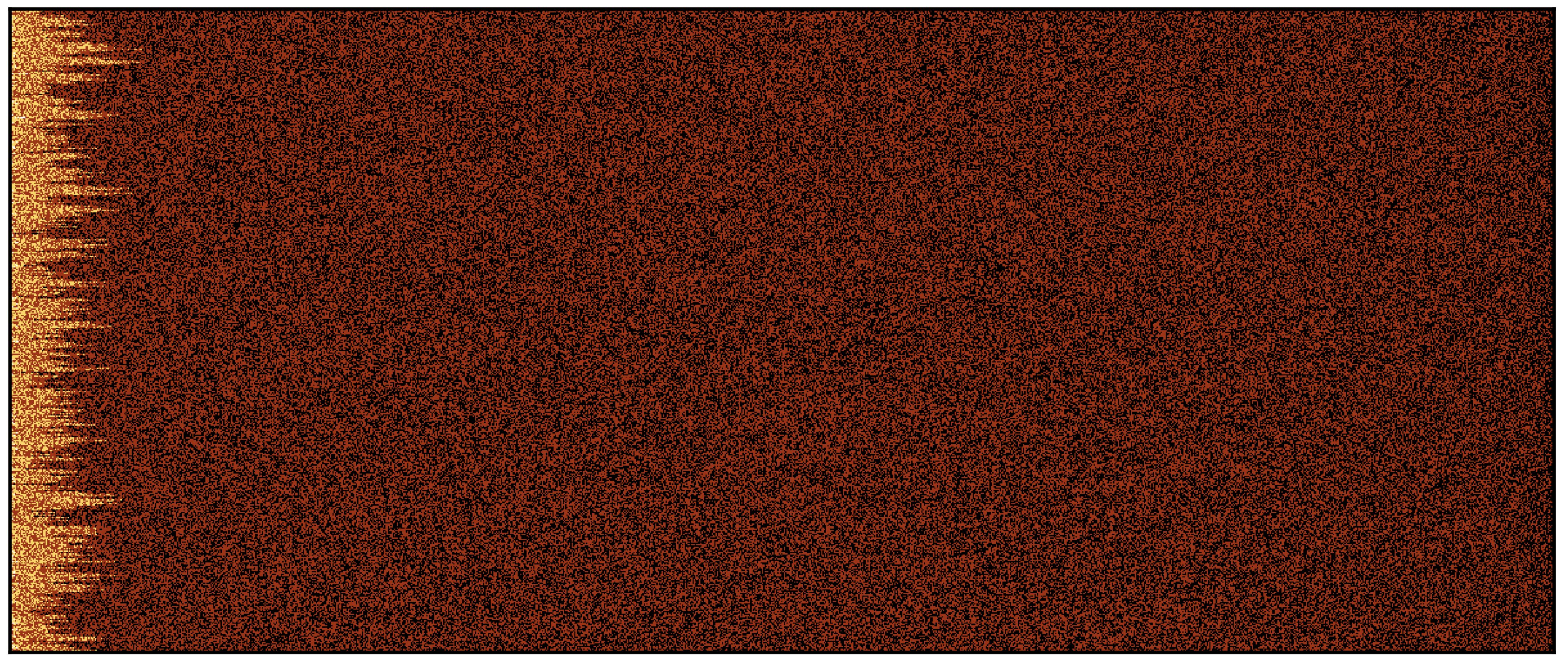}
    \caption{$\alpha=0.3, \beta=1, c=3$}
    \label{multi_tasep_30_100-c3-zero-init-dur2000}
\end{figure}

  \begin{figure}
\centering
        \includegraphics[width=6in]{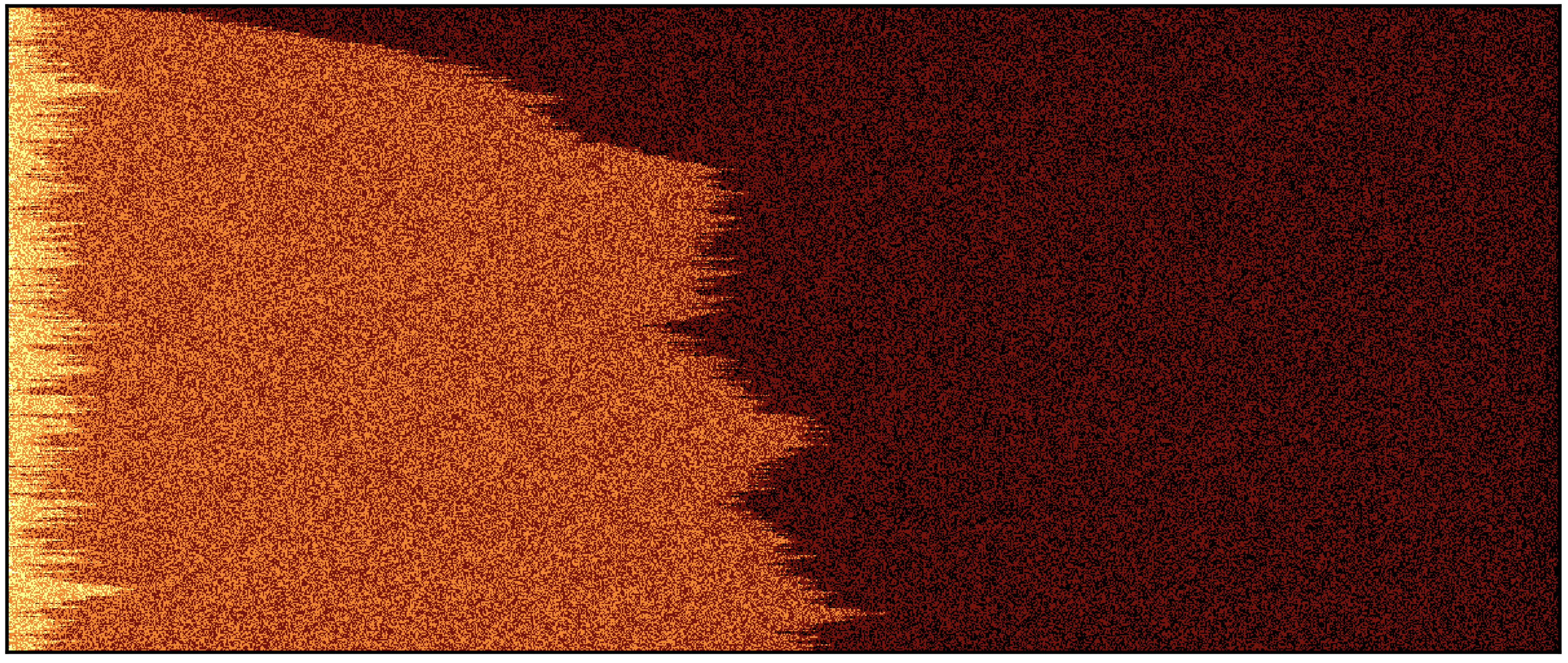}
    \caption{$\alpha=0.3, \beta=1, c=4$}
    \label{multi_tasep_30_100-c4-zero-init-dur2000}
\end{figure}

  \begin{figure}
\centering
        \includegraphics[width=6in]{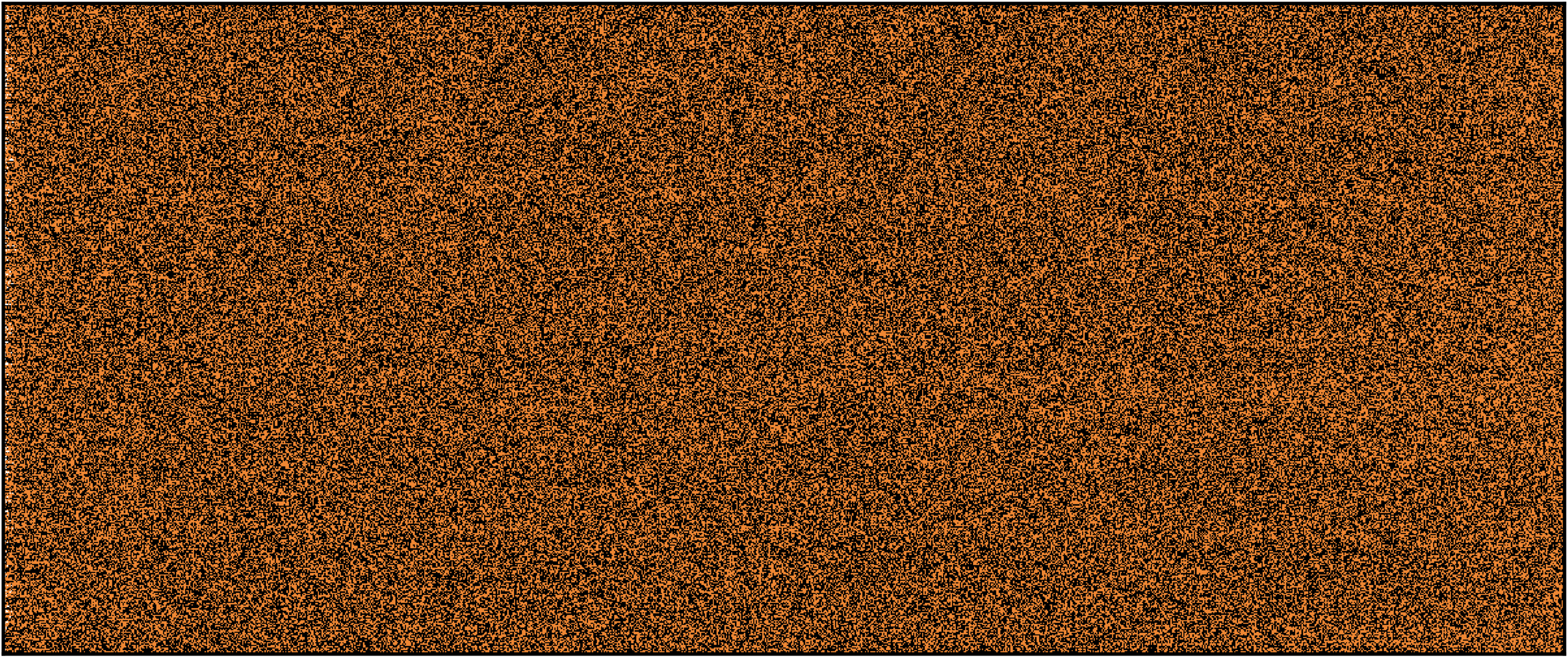}
    \caption{$\alpha=0.3, \beta=0.48, c=2$}
    \label{multi_tasep_30_48-c2-zero-init-dur2000}
\end{figure}

  \begin{figure}
\centering
        \includegraphics[width=6in]{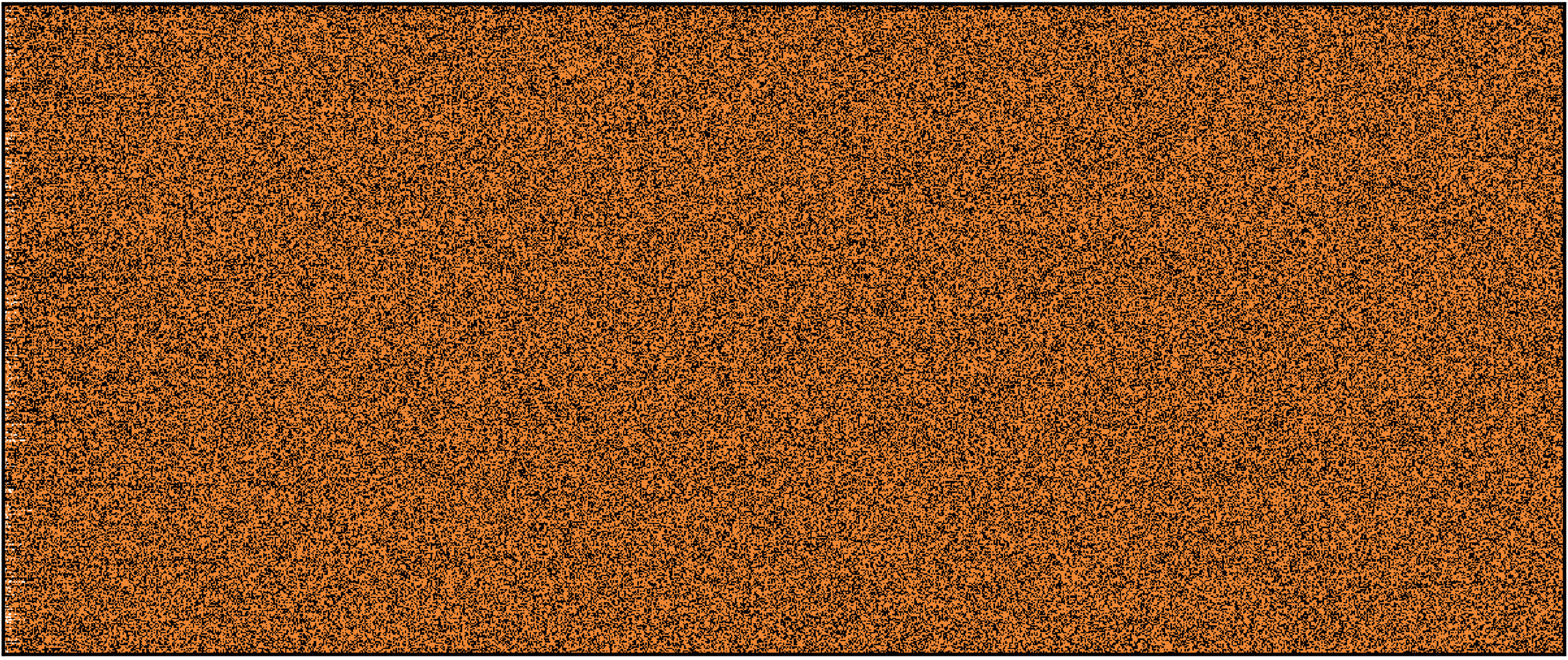}
    \caption{$\alpha=0.3, \beta=0.4, c=2$}
    \label{multi_tasep_30_40-c2-zero-init-dur2000}
\end{figure}

  \begin{figure}
\centering
        \includegraphics[width=6in]{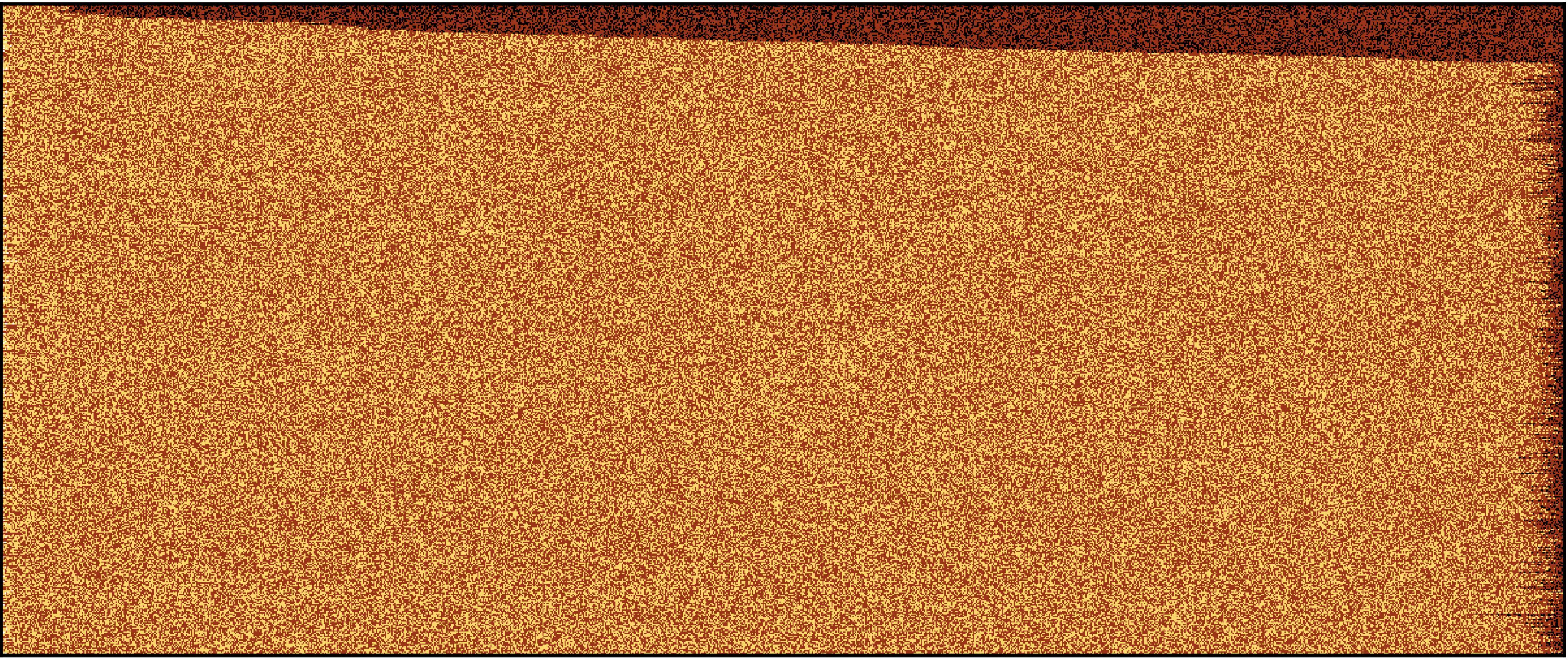}
    \caption{$\alpha=0.3, \beta=0.4, c=3$}
    \label{multi_tasep_30_40-c3-zero-init-dur2000}
\end{figure}

  \begin{figure}
\centering
        \includegraphics[width=6in]{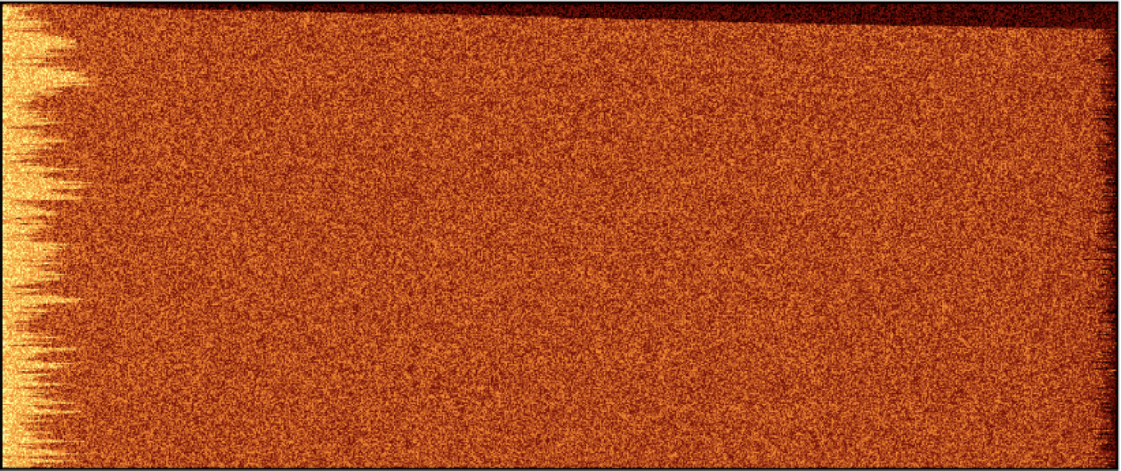}
    \caption{$\alpha=0.3, \beta=0.4, c=4$}
    \label{multi_tasep_30_40-c4-zero-init-dur2000}
\end{figure}

  \begin{figure}
\centering
        \includegraphics[width=6in]{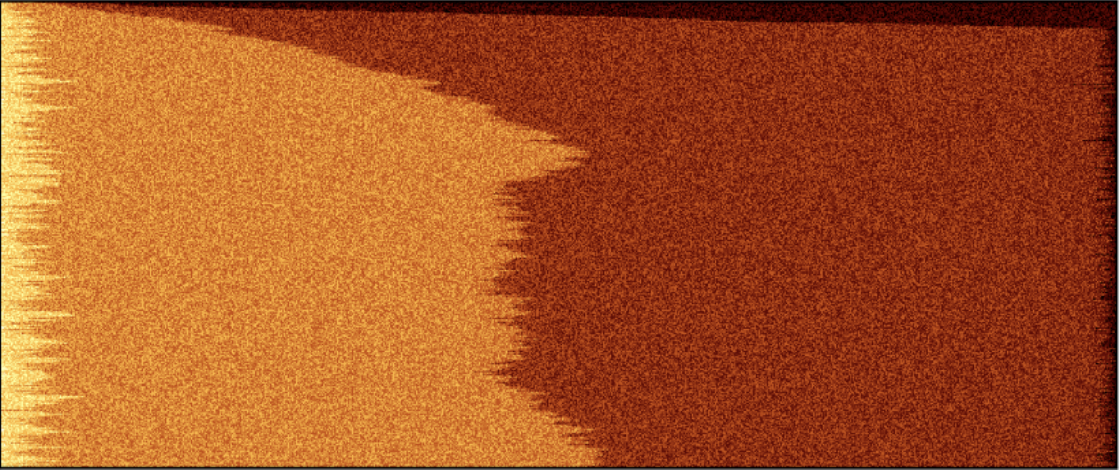}
    \caption{$\alpha=0.3, \beta=0.4, c=5$}
    \label{multi_tasep_30_40-c5-zero-init-dur2000}
\end{figure}

  \section{A general conjecture}
  \label{sec-obbc}

   \subsection{The conjecture and its motivation}
   
   Let us introduce the following definitions.
For $c\ge 1$, $\alpha > 0$ and $\beta>0$, let
 \beql{eq-max-flux-gen}
   \zeta(\alpha,\beta,c) \doteq \max_{1\le \ell \le c} \min\{\phi(\alpha,c-\ell+1), \phi(\beta,\ell)\}.
   \eeql
Denote by $L^*$ the set of $\ell$ where the maximum in \eqn{eq-max-flux-gen} is attained: 
   $$
   L^* \doteq \argmax_{1\le \ell \le c} \min\{\phi(\alpha,c-\ell+1), \phi(\beta,\ell)\}.
   $$
   It is easy to see that set $L^*$ necessarily has the form $L^*=\{\ell_1 \le \ell \le \ell_2\}$, and if   $\zeta(\alpha,\beta,c)<1/4$
   it has at most two elements.
   
   In Conjecture~\ref{conj-obbc}
   we will consider the following conditions:
   \beql{eq-case-a}
   \mbox{$L^*$ consists of a single element $\ell^*$ and $\phi(\alpha,c-\ell^*+1) \ne \phi(\beta,\ell^*)$},
   \eeql
   \beql{eq-case-b}
  \zeta(\alpha,\beta,c) = 1/4.
   \eeql
   Using properties of flux function $\phi(\alpha,c)$ (in Theorem~\ref{th-main-res} and  Lemma~\ref{lem-phi-props}), it is easy to observe the following.
   Condition \eqn{eq-case-a} necessarily implies that $\zeta(\alpha,\beta,c) < 1/4$.
   Condition \eqn{eq-case-b} 
    is equivalent to $c \ge c_\alpha + c_\beta -1$, and it necessarily implies that $\alpha > 1/4$, $\beta>1/4$, 
   and $L^* = \{c_\beta, \ldots, c-c_\alpha+1\}$. Conditions \eqn{eq-case-a} and \eqn{eq-case-b} cover a ``typical'' case 
   in the following sense: for any fixed $c$ and $\beta$, either \eqn{eq-case-a} or \eqn{eq-case-b} holds for all 
   $\alpha$ except a set of zero Lebesgue measure.
    
         \begin{conjecture} 
  \label{conj-obbc}
      Consider a finite system with the number of floors $c \ge 1$, and parameters $\alpha > 0$ and $\beta>0$. 
   Then, as $N\to\infty$, the following holds.
   
   (i) The limiting flux is equal to $\zeta(\alpha,\beta,c)$.
   
   (ii) Suppose condition \eqn{eq-case-a} holds. Let $\phi(\alpha,c-\ell^*+1) < \phi(\beta,\ell^*)$ to be specific (there is no loss of generality, because of particles/holes interchangeability). Then we have the following. The sequence of stationary distributions is such that there is 
only one large stable zone, on floor $\ell^*$, ``within which'' the distribution converges to $\Psi_{\ell^*-1} \nu_\gamma$,
   where the density $\gamma <1/2$ is determined by the flux, $\gamma(1-\gamma) = \phi(\alpha,c-\ell^*+1)= \zeta(\alpha,\beta,c)$.
   The left-side limit of the stationary distribution is $\ch_l = \Psi_{\ell^*-1} \cl(\alpha,c-\ell^*+1)$, which behaves like $\Psi_{\ell^*-1} \nu_\gamma$ at infinity.
   The right-side limit of the stationary distribution $\ch_r$ is a distribution with effective floor $c-\ell^*+1$,
   which behaves like $[\Psi_{c-\ell^*} \nu_{1-\gamma}]^{\updownarrow}$ at infinity; $\ch_r$ is equal to the (lifted up by $c-\ell^*$ floors) limit for the one-sided
   system with $\ell^*$ floors, the attempted (hole) arrival rate $\beta$, the attempted (hole) departure rate (from the other end) $\gamma$.
   
   (iii) Suppose condition \eqn{eq-case-b} holds.  (Recall that in this case, necessarily, $\alpha > 1/4$, $\beta>1/4$, $c \ge c_\alpha + c_\beta -1$, and
   $L^* = \{c_\beta, \ldots, c-c_\alpha+1\}$.)  Then we have the following. The sequence of stationary distributions is such that there are
   exactly $|L^*| = c +2 - c_\alpha - c_\beta$ large stable zones, of asymptotically equal size, 
  on floors $\ell = c_\beta, c_\beta + 1, \ldots, c-c_\alpha+1$, ``within which'' the distribution converges to $\Psi _{\ell-1} \nu_{1/2}$.
   The left-side asymptotic limit $\ch_l$ of the stationary distribution is $\Psi_{c-c_\alpha} \cl(\alpha,c_\alpha)$;
 the right-side asymptotic limit $\ch_r$ of the stationary distribution (from the point of view of holes) is
  $[\Psi_{c-c_\beta} \cl(\beta,c_\beta)]^{\updownarrow}$.
      \end{conjecture}

        The intuition for the conjecture is as follows. As a thought experiment,  suppose that, as $N\to\infty$, there is only one large stable zone, on a floor $\ell^*$.
    Then the left-side limit has effective floor $\ell^*$. So, particles moving from the left, ``cascade down'' and, eventually, after $O(1)$ time end up on floor $\ell^*$, on which the process is just like standard TASEP (lifted up by $\ell^*-1$ floors). If the attempted ``departure''  rate of this process (from some imaginary end site far on the right) would be $1$, then the left side limit would be $\Psi_{\ell^*-1} \cl(\alpha, c-\ell^*+1)$ and have flux $\phi(\alpha, c-\ell^*+1)$;
    let us further assume that it would behave like $\nu_\gamma$ at infinity, where the bulk particle density $\gamma \le 1/2$ is determined by the flux,
    namely $\gamma(1-\gamma) = \phi(\alpha, c-\ell^*+1)$. Analogously, the right-side limit (from the point of view of holes) has effective floor
    $c-\ell^*+1$; if the attempted ``departure''  rate of holes (from some imaginary end site far on the left) would be $1$, then the right-side limit would be
    $[\Psi_{c-\ell^*} \cl(\beta, \ell^*)]^\updownarrow$ and have flux $\phi(\beta, \ell^*)$; and let us further assume that it would behave like $\nu_{\gamma'}$ at infinity, where the bulk hole density $\gamma' \le 1/2$ is determined by the flux,
    namely $\gamma'(1-\gamma') = \phi(\beta, \ell^*)$. The particles of the left-side process ``collide'' with the holes of the right-side process
    at some ``point'' within the floor $\ell^*$ zone. Suppose, finally, that the flux $\phi(\alpha, c-\ell^*+1)$ of particles from the left  is smaller than the flux of holes $\phi(\beta, \ell^*)$
    from the right, that is $\gamma < \gamma'$. Then, it is reasonable to assume that the limit of stationary distribution within the bulk of 
    the large stable floor-$\ell^*$ zone is same as the limit of stationary distribution within the bulk of the floor-$1$ zone of the single-floor process with 
    the attempted arrival rate $\gamma$ at the left and the attempted departure rate $\gamma'$ on the right. As stated
    is Section~\ref{sec-left-right}, the latter distribution is $\nu_\gamma$, given $\gamma < \gamma'$, and the limiting flux is 
    $\gamma(1-\gamma)= \min\{\gamma(1-\gamma),\gamma'(1-\gamma')\}$. Therefore, in our multi-floor system,
    we should expect the bulk particle density within the large stable floor-$\ell^*$ zone to be $\gamma$ and the flux to be
    $\phi(\alpha, c-\ell^*+1)=\min\{\phi(\alpha, c-\ell^*+1),\phi(\beta, \ell^*)\}$.
    
    Therefore, Conjecture~\ref{conj-obbc} states that, when $N$ is large, the steady-state is such that the system ``selects'' a floor $\ell^*$ on which
    a large stable zone is formed, so that the flux is maximized among all possible $\ell^*$. The ``optimal'' floor $\ell^*$ may be non-unique.
   Conjecture~\ref{conj-obbc} further states that, for a ``typical'' setting of parameters
    $c$, $\alpha$, $\beta$ (when either  \eqn{eq-case-a} or  \eqn{eq-case-b} holds), 
    either the limiting flux is strictly less than $1/4$ and the optimal $\ell^*$ is unique (and then there is only one large stable zone) or
    the limiting flux is equal to $1/4$ and a well-defined number of equal-size large stable zones form.

It is readily seen that Theorem~\ref{th-main-res}(ii), Proposition~\ref{prop-single-floor}, and Theorem~\ref{th-special-max-flux} 
are special cases of Conjecture~\ref{conj-obbc}.

       \subsection{Simulation results' interpretation in terms of the conjecture}

    All our simulation results (including those not presented in this paper) appear to conform with Conjecture~\ref{conj-obbc}. We now discuss 
    our simulation results again, and 
    interpret them in terms of the conjecture. 
    
    When $\alpha=\beta=1$, we have $c_\alpha=c_\beta=1$, and therefore the limiting flux is $1/4$ and exactly $c$ equal-sized large stable zones should form.
    That is what we see in Figures~\ref{multi_tasep_100_100-c2-zero-init-dur2000} and \ref{multi_tasep_100_100-c3-zero-init-dur2000}.
   
When   $\alpha=0.45$, $\beta=1$, we have $c_\alpha=2$ and $c_\beta=1$. Therefore, when $c\ge 2$,
 the limiting flux is $1/4$ and exactly $c-1$ equal-sized large stable zones should form on floors $1,\ldots,c-1$.
 Figures~\ref{multi_tasep_45_100-c2-zero-init-dur2000} and \ref{multi_tasep_45_100-c3-zero-init-dur2000} demonstrate that.
 
In the case [$\alpha=0.45$, $\beta=0.6$, $c=2$], just like in the case [$\alpha=0.45$, $\beta=1$, $c=2$], we have $c_\alpha=2$ and $c_\beta=1$.
Therefore, in both cases, the only large stable zone should be on floor $1$ and the left-side limit should be $\cl(0.45,2)$.
This is what we see in Figures~\ref{multi_tasep_45_60-c2-zero-init-dur2000} and \ref{multi_tasep_45_100-c2-zero-init-dur2000}.

If $\alpha=0.3$, we have $c_\alpha=3$. Recall that $c_\beta=1$ for $\beta\ge 1/2$.
Then, in the case [$\alpha=0.3$, $\beta=1$, $c=2$] we should see the only large stable zone on floor $1$ and the left-side limit  $\cl(0.3,2)$, with flux $< 1/4$.
The simulation in Figure~\ref{multi_tasep_30_100-c2-zero-init-dur2000} conforms to that, and indicates that flux $\phi(0.3,2)\approx 0.46(1-0.46)$.
When we increase the number of floors to $c\ge 3$,
we should see the left-side limit $\cl(0.3,c)$, flux $1/4$, and $c-2$ equal-sized large stable zones on floors $c-2, \ldots, 1$. 
The simulations in Figures~\ref{multi_tasep_30_100-c3-zero-init-dur2000}-\ref{multi_tasep_30_100-c4-zero-init-dur2000} indeed show that.

The case [$\alpha=0.3$, $\beta=0.48$, $c=2$] is such that $c_\alpha=3$ and $c_\beta=2$. According to our simulation of the case 
[$\alpha=0.3$, $\beta=1$, $c=2$], we have $\phi(0.3,2)\approx 0.46(1-0.46)$. We also know that $\phi(0.48,1) = 0.48(1-0.48)$. Given these numerical values, it is easy to check that set $L^*$ has unique element $\ell^*=1$ and we are within conditions of Conjecture~\ref{conj-obbc}(ii). Therefore, 
there should be only one large stable zone, on floor $1$,
the left-side limit should be $\cl(0.3,2)$, and flux should be $\phi(0.3,2)\approx 0.46(1-0.46)$; this is the same situation as in the case [$\alpha=0.3$, $\beta=1$, $c=2$]. This is what we actually observe -- see Figure~\ref{multi_tasep_30_48-c2-zero-init-dur2000} and compare it to 
Figure~\ref{multi_tasep_30_100-c2-zero-init-dur2000}.
 
 Now consider the case [$\alpha=0.3$, $\beta=0.4$, $c=2$]. We have $c_{0.3}=3$, $c_{0.4}=2$, $\phi(0.3,2)\approx 0.46(1-0.46)$. 
 But, unlike in the case [$\alpha=0.3$, $\beta=0.48$, $c=2$], now $\phi(0.4,1) = 0.4(1-0.4) < \phi(0.3,2)$.
 We are, again, in the conditions of Conjecture~\ref{conj-obbc}(ii), with $\ell^*=1$. But, now the left-side limit is not $\cl(0.3,2)$.
 Instead, the right-side limit is $[\Psi_1 \cl(0.4,1)]^\updownarrow=[\Psi_1 \nu_{0.4}]^\updownarrow$. Therefore, there should be only one large stable zone, on floor $1$,
 with bulk particle density $1-0.4$ (i.e., holes' density $0.4$) and flux $0.4(1-0.4)$. That is what actually seen in simulation
 -- see Figure~\ref{multi_tasep_30_40-c2-zero-init-dur2000}.
 
 Next, consider case [$\alpha=0.3$, $\beta=0.4$, $c=3$]. Once again, $c_{0.3}=3$, $c_{0.4}=2$, $\phi(0.3,2)\approx 0.46(1-0.46)$.
 We know that $2 \le c_{0.4}$ and $\phi(0.4,2) > \phi(0.3,2)$. These numerical values imply that
we are in the conditions of Conjecture~\ref{conj-obbc}(ii), with $\ell^*=2$ and $\phi(0.3,\ell^*) < \phi(0.4,\ell^*)$. 
Therefore, we should see the only large stable zone on floor $2$, with bulk particle density $\approx 0.46$ (on this floor),
and with the left-side distribution limit $\Psi_1 \cl(0.3,2)$. We do observe all these properties in simulation -- 
see Figure~\ref{multi_tasep_30_40-c3-zero-init-dur2000}. If we increase the number of floors to $c=4$ -- this is case [$\alpha=0.3$, $\beta=0.4$, $c=4$] -- 
we observe (Figure~\ref{multi_tasep_30_40-c4-zero-init-dur2000}) one large stable zone on floor $2$, with density about $1/2$, and flux close to $1/4$. 
This indicates that $c_{0.4}=2$, and then $c_{0.3}+c_{0.4}=3+2 = c+1$, and then we are now in 
conditions of Conjecture~\ref{conj-obbc}(iii). This is further confirmed by the case [$\alpha=0.3$, $\beta=0.4$, $c=5$], where we should -- and do
(Figure~\ref{multi_tasep_30_40-c5-zero-init-dur2000}) -- observe the formation of two large stable zones on floors $2$ and $3$, of approximately equal size.

 \section{Discussion and future research directions}
  \label{sec-discussion}
  
  \edit{
  We now briefly compare our results to those in \cite{CEKT16,SJHW11}, for the different multi-lane models described in Section~\ref{sec-model-motivations}. Our model could be viewed as a limiting case of the model in \cite{CEKT16} when: $\alpha_c=\alpha$, $\alpha_m=0$ for $m<c$;
  $\beta_1=\beta$, $\beta_m=0$ for $m>1$; the migration rates are $d_{m,\ell} = \infty$ for $m>\ell$, and $d_{m,\ell} = 0$ for $m<\ell$. In other words, 
 at the left boundary particles can only arrive into the highest lane (floor) $c$, at the right boundary they can only depart from the lowest lane (floor) $1$,
 and any particle at any site instantly migrates to the lowest vacant lane (floor). Paper \cite{CEKT16} studies a mean-field based hydrodynamic model -- a set of coupled differential equations, describing the evolution of
 particle densities in each lane. The paper focuses on stationary solutions being equilibrated plateaux, which correspond to the sets of lane particle densities $\{\rho_m\}$, that are stable with respect to inter-lane migration dynamics. The analysis further assumes that the ``left particle reservoir densities'' (the arrival rates $\{\alpha_m\}$) and ``right particle reservoir densities'' (the set $\{1-\beta_m\}$) are also equilibrated (stable). However, if the model in present paper is cast as a limiting case of the model in \cite{CEKT16} (with $\alpha_m$ and $\beta_m$ specified above), the boundary conditions $\{\alpha_m\}$ and $\{1-\beta_m\}$
 are {\em not} equilibrated. In this case, as \cite{CEKT16} states, the relation between the equilibrated densities
of the plateaux in the bulk and the boundary conditions is not known. Consequently, for our model, the hydrodynamic approximation approach in \cite{CEKT16} does not lead to conclusions beyond the fact that any equilibrated plateau densities are such that, for some floor $m$,
$\rho_\ell=0$ for $\ell>m$ and $\rho_\ell=1$ for $\ell<m$, which corresponds to the process ``living on one of the floors;'' 
but this fact trivially follows from the structure of our model, and does not tell {\em how the system chooses the floor(s)} in steady-state and {\em what is 
the particle density on those floors} -- which our results do address.
}
 
  \edit{
 Paper \cite{SJHW11} considers a strongly asymmetric two-lane system. The key element of the analysis in \cite{SJHW11} is that the mean-field assumption (basically, that the site states are i.i.d.) 
 leads to a non-trivial relation between the distribution of a site state and the flux at the boundaries. 
 For the model in the present paper, such a mean-field assumption does not lead to any non-trivial results, even approximate.
 Indeed, it leads to two possible cases: the process living either on the first floor or on the second floor. The former case would imply that the steady-state flux is simply equal to the attempted arrival rate $\alpha$, which is obviously false. The latter case would imply, for example when $0< \alpha < 1/2$ and $\beta=1$, that
 the flux is $\alpha(1-\alpha)$, which is also false. 
 Thus, as stated in Section~\ref{sec-model-motivations}, the methods in previous work do not appear to lead to 
 non-trivial conclusions for our model.
 }

There are several interesting directions of future research into the multi-floor TASEP model, introduced in this paper. Among them: 
computing/estimating the values of the thresholds  $\alpha^*_i$ for $i\ge 2$; 
formal proof of Conjecture~\ref{conj-obbc};
a formal analysis  of the dynamics (both transient and steady-state) of the fronts (boundaries) between neighboring zones.

{\bf Data availability statement}

All data is included in the paper text.

\bibliographystyle{abbrv}


\end{document}